\documentclass[reqno]{amsart}
\usepackage{amssymb}
\usepackage{xcolor}
\usepackage[pagewise]{lineno}
\usepackage{soul}

\usepackage{graphicx}

\newtheorem{Lemma}{Lemma}[section]
\newtheorem{Theorem}[Lemma]{Theorem}
\newtheorem{Proposition}[Lemma]{Proposition}
\newtheorem{Corollary}[Lemma]{Corollary}

\newtheorem{Definition}[Lemma]{Definition}

\newcommand{\removedEinstein}[1]{}

\newcommand{\extension}[1]{} 

\newcommand{\generalizations}[1]{}

\newcommand{\hiddenfootnote}[1]{}

\def\S{{\mathcal S}}

\newcommand{\R}{{\mathbb R}}

\newcommand{\N}{{\mathbb N}} 

\def\tilde{\widetilde}
\def \bfo {\begin {eqnarray*} }
\def \efo {\end {eqnarray*} }
\def \ba {\begin {eqnarray*} }
\def \ea {\end {eqnarray*} }
\def \beq {\begin {eqnarray}}
\def \eeq {\end {eqnarray}}

\def \dist {\hbox{dist}}

\def \det {\hbox{det}}

\def\bra{\langle}

\def \p {\partial}

\def\cR{{\mathcal R}}
\def\cS{{\mathcal S}}

\def\Met{\mathcal{M}et}

\begin{document}
\title[]{Reconstruction of a compact manifold from the scattering data of 
internal sources}

\author[M. Lassas]{Matti Lassas $^\square$}

\author[T. Saksala]{Teemu Saksala $^{\diamond,\: \ast}$}

\author[H. Zhou]{Hanming Zhou $^\dagger$}

\let\thefootnote\relax\footnote{ $^\square$ Department of Mathematics and Statistics, University of Helsinki, Finland, 
\\
$^\diamond$  Department of Mathematics and Statistics, University of Helsinki, Finland; Department of computational and applied mathematics, Rice University, USA
\\
$^\dagger$ Department of Pure Mathematics and Mathematical Statistics, University of Cambridge, UK; Department of Mathematics, University of California Santa Barbara, Santa Barbara, CA 93106, USA, 
\\ 
$^\ast$ \textbf{teemu.saksala@helsinki.fi, teemu.saksala@rice.edu}}

\begin{abstract}
Given a smooth non-trapping compact 
manifold with strictly convex boundary, we consider an inverse problem of reconstructing the manifold from the scattering data initiated from internal sources. These data consist of the exit directions of geodesics that are emaneted from interior points of the manifold. We show that under certain generic assumption of the 
metric, the scattering data 
measured on the boundary determine the Riemannian manifold up to isometry.
\end{abstract}

\maketitle
\tableofcontents
\section{Introduction, problem setting and main result}
In this paper we consider an inverse problem of reconstructing a Riemannian manifold $(\overline M,g)$ from geodesical data that correspond to the following theoretical measurement set up. Suppose that in a domain $M$ there is a large amount $m$ of point sources, at points $x_1,x_2,\dots,x_m$, sending continuously light (or other high frequency waves) at different frequencies $\omega_j$. Such point sources are observed on the boundary as bright points. Assume that on the boundary an observer records at the point $z\in \p M$ the exit direction of the light from the point source $x_j$, that is, an observer can see at the point $z$ the directions of geodesics coming from $x_j$ to $z$. When the observer moves along the boundary, all existing directions of geodesics coming from the point sources to the boundary are observed. We emphasize that only the directions, not the lengths (i.e., travel times) are recorded. When $m$ becomes larger, i.e., $m\to \infty$,
we can assume that the set of the point sources $\{x_j\}$ form a dense set in $M$. 


Let $(N,g)$ be an $n$-dimensional closed smooth Riemannian manifold, where $n\geq 2$. Suppose that $M \subset N$ is open and has a boundary $\p M \subset N$ that is a smooth $(n-1)$-dimensional submanifold of $N$. We also assume that $\p M$ is strictly convex, meaning that the second fundamental form of $\p M$ (as a submanifold) is positive definite. 

Let $SN\subset TN$ be the unit tangent bundle of the metric $g$. Given $(p,\xi)\in SN$, it induces a unique maximal geodesic $\gamma_{p,\xi}:\R \rightarrow N$ with $\gamma_{p,\xi}(0)=p,\, \dot{\gamma}_{p,\xi}(0)=\xi$. We denote  $S_pN=\{\xi\in T_pN;\ \|\xi\|_g=1\}$. We will use the notation $\pi:TN \to N$ for the canonical projection to the base point, $\pi(x,\xi)=x$.


We define the first exit time function
\begin{equation}
\label{eq:exit_time_function}
\tau_{exit}(p,\xi)=\inf \{t>0; \gamma_{p,\xi}(t)\in N\setminus \overline{M}\},  \:(p,\xi) \in S \overline{M},
\end{equation}
where $S \overline{M}$ is the unit tangent bundle of $\overline{M}$. We assume that 
\begin{equation}
\label{eq_exit time is bounded}
\tau_{exit}(p,\xi)< \infty, \: \hbox{ for all }(p,\xi) \in S \overline{M}.
\end{equation}
This means that $\overline{M}$ is non-trapping.

For each point $p\in  \overline{M}$ we define a \textit{scattering set of the point source $p$} 
\begin{equation}
\label{point source set}
\begin{array}{lll}
R_{\p M}(p)&=&\{(q,\eta^T)\in T \p M ; \textrm{ there exist }\,\xi\in S_pN 
\\
& &\textrm{and } t\in [0,\tau_{exit}(p,\xi)] \textrm{ such
that }
\\
& & q=\gamma_{p,\xi}(t), \eta=\dot  \gamma_{p,\xi}(t)\} \in 2^{T\p M}.
\end{array}
\end{equation}
Here $\eta^T \in T\p M$ is the tangential component of $\eta \in \p SM$ and $2^S$ means the power set of set $S$. See Figure \eqref{fig:data}. 
\begin{figure}[h]
 \begin{picture}(100,150)
  \put(-60,-10){\includegraphics[width=200pt]{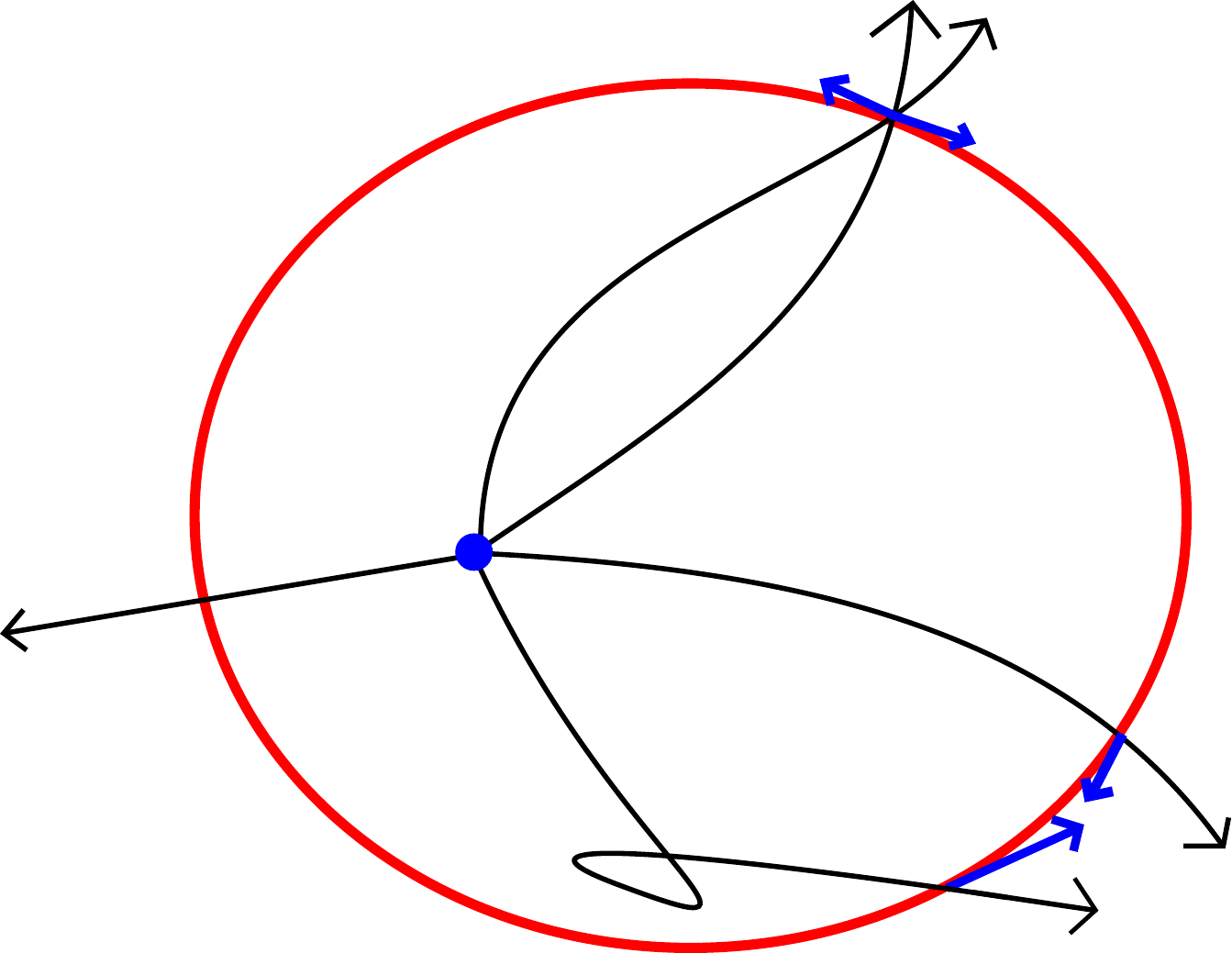}}
  \put(-50,70){$\p M$}
  \put(10,60){$p$}
   \end{picture}
\caption{Here is a schematic picture about our data $R_{\p M}(p)$, where the point $p$ is the blue dot. Here the black arrows are the exit directions of geodesics emitted from $p$ and the blue arrows are our data.}
\label{fig:data}
\end{figure}

Let
\ba
R_{\p M}(M)=\{R_{\p M}(p);\ p\in M\} \subset 2 ^{T\p M}.
\ea
%
%
%
%
Consider the following collection

\begin{equation}
(\p M ,R_{\p M}(M)).
\label{geodesic measurement data1}
\end{equation} 
We call this collection the {\it scattering data of internal sources} that depends on the Riemannian manifold $(\overline{M},g|_{\overline{M}})$. We emphasize the connection of data \eqref{geodesic measurement data1} to the scattering data (see \cite{michel1981rigidite}), that is also considered in the Section \ref{subsec_Boundary} of this paper.  From now on we will use a short hand notation
$$
g:=g|_{\overline M}=i^\ast g, \: i:\overline{M} \hookrightarrow N.
$$
Here $i$ is an embedding defined by $i(x)=x, \: x \in \overline{M}$. The inverse problem considered in this paper is the determination of the Riemannian manifold $(\overline{M},g)$ from the data \eqref{geodesic measurement data1}. More precisely, this means the following: Let $(N_j,g_j), \: j=1,2$ and  $M_j$ be similar to $(N,g)$ and $M$.  We say that the scattering data of internal sources of $(\overline{M}_1,g_1)$ is equivalent to that of $(\overline{M}_2,g_2)$, if
\begin{eqnarray}
\label{eq:collar_neib1}
& & \hbox{there exists a diffeomorphism $\phi: \p M_1 \rightarrow \p M_2$ such that}, 
\\
\label{eq:equivalenc_of_data}
& &
\{D\phi (R_{\p M_1}(q));\ q\in M_1\}= \{R_{\p M_2}(p);\ p\in M_2\}.
\end{eqnarray}
Here $D\phi:  T \p M_1 \to  T\p M_2$ is the differential of $\phi$, i.e., the tangential mapping of $\phi$. 

\medskip

We denote the set of all $C^\infty$-smooth Riemannian metrics on $N$ by $\mathcal{M}et(N)$ (we will assign the smooth Whitney topology (see \cite{hirsch2012differential}) on $\mathcal{M}et(N)$ to make it into a topological space), it turns out that there exists a generic subset (a set that contains an countable intersection of open and dense sets) $G\subset\mathcal{M}et(N)$ such that for $g\in G$, the manifold $(\overline{M},g)$ is determined by the data \eqref{geodesic measurement data1} up to an isometry. 

Given $g\in \Met(N),\, p,q\in N$ and $\ell>0$, we denote the number of $g$-geodesics connecting $p$ and $q$ of length $\ell$ by $I(g,p,q,\ell)$. Define 
$$I(g):=\sup_{p,q,\ell}I(g,p,q,\ell).$$
In Theorem 1.2 of \cite{kupka2006focal} it is shown that there exists a generic set $G \subset \Met(N)$, such that for all $g\in G$, 
\begin{equation}
\label{eq:KPP_property}
I(g)\leq 2 n+2.
\end{equation}

Next we define the collection of admissible Riemannian manifolds. Denote
\begin{equation}
\label{eq:admissible_metrics}
\mathcal{G}=\{(N,g); N \hbox{ is a connected, closed, smooth Riemannian $n$-manifold; } g \hbox{ satisfies } \eqref{eq:KPP_property}\}.
\end{equation} Our main result is as follows,

\begin{Theorem}
\label{th:main}
Let $(N_i,g_i) \in \mathcal{G}, \: i=1,2$ be a smooth, closed, connected Riemannian $n$-manifold $n\geq 2$, $M_i \subset N_i$ be an open set with smooth strictly convex boundary with respect to $g_i$. Suppose that $(\overline{M}_i,g_i)$, $i=1,2$  is non-trapping, in the sense of \eqref{eq_exit time is bounded}.

If  properties \eqref{eq:collar_neib1}--\eqref{eq:equivalenc_of_data} hold, then $(\overline{M_1},g|_{\overline{M}_1})$ is Riemannian isometric to $(\overline{M}_2,g_2|_{\overline{M}_2})$.
\end{Theorem}


The auxiliary results needed to prove the main theorem will be stated and proven in such a way that we will use the data \eqref{geodesic measurement data1} to reconstruct an isometric copy of $(\overline{M},g)$. Therefore we formulate also the following theorem.
\begin{Theorem}
\label{th:reconstruction}
Let $(N,g) \in \mathcal{G}$ be a smooth, closed, connected Riemannian $n$-manifold $n\geq 2$, $M \subset N$ be an open set with smooth strictly convex boundary with respect to $g$. Suppose that $(\overline{M},g)$,  is non-trapping, in the sense of \eqref{eq_exit time is bounded}. 

If the data \eqref{geodesic measurement data1}  is given we can reconstruct a isometric copy of $(\overline{M},g)$.
\end{Theorem}

We want to underline that in the most parts  the proofs for Theorems \ref{th:main} and \ref{th:reconstruction} go side by side! Therefore we will formulate the lemmas and propositions, given in Section 2 and later, in a way having this dual goal in mind.

\color{black}
\subsection{Outline of this paper}
The proof of Theorem \ref{th:main} is contained in Sections 2 to 5. In Section 2 we recall some properties of the strictly convex boundaries and the first exit time function $\tau_{exit}$. We will also give a generic property related to the generic property \eqref{eq:KPP_property} which guarantees that the scattering set $R_{\p M}(p)$ is related to a unique point $p \in M$. The main result of Section 3 is to show that the data $\eqref{geodesic measurement data1}$ determines uniquely the topological structure of the manifold $\overline{M}$. Section 4 is devoted to the smooth structure of $\overline M$ and the final section shows that the data $\eqref{geodesic measurement data1}$ determines the Riemannian structure of $(\overline{M}, g)$.
\subsection{Previous literature}

The research topic of this paper is related to many other inverse problems. 

\subsubsection{Boundary and scattering rigidity problems}
\label{subsec_Boundary}

The scattering data of internal sources \eqref{geodesic measurement data1} considered in our paper takes into account the geodesic rays emitted from all the points of $M$, though we only know the exit positions and directions of these geodesics on the boundary. Instead of the  scattering data of internal sources, the {\it scattering relation} is defined as follows
$$L_g:\p_+SM\to \p_-SM,\quad L_g(x,\xi):=\bigg(\gamma_{x,\xi}(\tau_{exit}(x,\xi)),\dot\gamma_{x,\xi}(\tau_{exit}(x,\xi))\bigg),$$
where 
$$\p_{\pm}SM:=\{(x,\xi)\in \p SM; \pm\langle \xi,\nu(x)\rangle_g\leq 0\}$$
are the sets of inward and outward pointing vectors on the boundary $\p M$. Here $\nu$ is the outward pointing unit normal to $\p M$. Notice that $M$ is non-trapping, thus $L_g$ is well-defined. We will see later that scattering data of internal sources $(\p M, R_{\p M}(M))$ determines the scattering relation $L_g$. The scattering rigidity problem asks: if $L_g=L_{\tilde g}$, does it implies that $\tilde g=\psi^* g$ with $\psi:\overline M\to \overline M$ a diffeomorphism fixing the boundary?  

The scattering rigidity problem is closely related to the boundary rigidity problem, which is concerned with the determination of the metric $g$ (up to a diffeomorphism fixing the boundary) from its boundary distance function $d_g: \p M\times \p M\to \mathbb R$. Notice that given $x,y\in \p M$, $d_g(x,y)$ equals the length of the distance minimizing geodesic connecting $x$ and $y$. See \cite{croke2004rigidity, stefanov2008boundary, uhlmann2016journey} for recent surveys on these topics. In particular, the scattering rigidity problem and boundary rigidity problem are equivalent on simple manifolds. A compact Riemannian manifold $M$ is
simple if the boundary $\p M$ is strictly convex and any two points can be joined by a
unique distance minimizing geodesic. Michel \cite{michel1981rigidite} conjectured that simple manifolds are boundary distance rigid, so far this is known for simple surfaces \cite{pestov2005two}. More recently, boundary rigidity results are established on manifolds of dimension $3$ or larger that satisfy certain global convex foliation condition \cite{stefanov2016boundary, stefanov2017local}. In the current paper, since the scattering data of internal sources contains more information than the scattering relation, we can deal with more general geometry.

It is worth mentioning that the problem of recovering Riemannian 
manifolds by the length data of geodesic rays emitted from internal 
sources was considered in \cite{ Katchalov2001, kurylev1997multidimensional, pestov2015inverse}.

\subsubsection{The rigidity of broken geodesic flows}

Another related inverse problem is concerned with determining a manifold $(\overline M ,g)$ from
the broken geodesic data, which consists of the initial and
the final points and directions, 
and the total length, of the broken geodesics. To define the data we first set up the notations for a broken geodesic
$$
\alpha_{x,\xi_+,z,\eta}(t)=\left\lbrace \begin{array}{l}
\gamma_{x,\xi_+}(t), \: t < s, \: (x,\xi_+) \in \p_+SM \\
\gamma_{z,\eta}(t-s), \: t \geq s, \: (z,\eta) \in SM,
\end{array} \right.
$$
where $z=\gamma_{x,\xi_+}(s)$.  Denote the length of the curve $\alpha_{x,\xi_+,z,\eta}$ by $\ell(\alpha_{x,\xi_+,z,\eta})$, the \textit{Broken scattering relation} is 
\begin{equation*}
\begin{array}{lll}
\mathcal{R} & =& \{((x,\xi_+),(y,\xi_-),t)\in \p_+SM\times \p_-SM\times \R_+; \\
& & t = \ell(\alpha_{x,\xi_+,z,\eta})\: \textrm{ and } (y,\xi_-)=(\alpha_{x,\xi,z,\eta}(t),\partial_t \alpha_{x,\xi,z,\eta}(t)), \\
&& \textrm{ for some } (z,\eta)\in SM\}.
\end{array}
\end{equation*}
In \cite{kurylev2010rigidity} the authors show that the broken scattering data $(\p M, \mathcal{R})$ determines the manifold $(\overline M,g)$ uniquely up to isometry, if dim $\overline M \geq 3$. In particular, there is no restrictions on the geometry of the manifold. 

\subsubsection{Inverse boundary value problem for the wave equations}
Next we will present a third well known inverse problem, where the data is given in the boundary of the manifold. Consider the following initial/boundary value problem for the Riemannian wave equation
\begin{equation}
\label{eq:initial/boundary value problem}
\begin{array}{l}
\p_t^2 w - \Delta_g w = 0 \quad \textrm{in } (0, \infty) \times M,
\\
w|_{\R \times \p M} = f, 
\\
w|_{t = 0} = \p_t w|_{t= 0} = 0, \: f\in C_0^\infty((0,\infty)\times \p M)
\end{array} 
\end{equation}
where $\Delta_g$ is the Laplace-Beltrami operator of metric tensor $g$. It is well known that for every $f\in C_0^\infty((0,\infty)\times \p M)$ there exists a unique $w^f \in C^\infty((0,\infty)\times \overline M)$ that solves \eqref{eq:initial/boundary value problem} (see for instance \cite{Katchalov2001} Chapter 2.3). Thus the Dircihlet-to-Neumann operator 
$$
\Lambda_g:C_0^\infty((0,\infty)\times \p M) \to C^\infty((0,\infty)\times \p M), \: \Lambda_gf(t,x)=\langle \nu(x), \nabla_g w^f(t,x) \rangle_g
$$ 
is well defined and an inverse problem to \eqref{eq:initial/boundary value problem} is to reconstruct $(\overline{M}, g)$ from the data
\begin{equation}
\label{eq:wave_data}
(\p M, \Lambda_g).
\end{equation}
One classical way to solve this problem is the\textit{ Boundary Control method} (BC). This method was first developed by Belishev for the acoustic wave equation on $\R^n$ with an isotropic wave speed
\cite{belishev1987approach}. A geometric version of the method, suitable when the wave speed is given by a Riemannian metric tensor as presented here, was introduced by Belishev and Kurylev \cite{belishev1992reconstruction}. The BC-method is used to determine the collection 
$$
\hbox{BDF}(M):=\{d_g(p,\cdot)|_{\p M}; \: p \in M\}
$$  
of \textit{boundary distance functions}. We emphasize the connection between our data $R_{\p M}(M)$ and $\hbox{BDF}(M)$, that is for any $p \in M$ and  for every $z \in \p M,$ it holds that holds
$$
\nabla_{\p M}( d_g(p,\cdot)|_{\p M})\bigg|_z\in R_{\p M}(p),
$$ 
if $d_g(p, \cdot)$ is smooth at $z$. Above $\nabla_{\p M}$ is the gradient of a Riemannian manifold $(\p M, g|_{\p M})$. One significant difference is that $BDF(M)$ contains only information about the velocities of the short geodesics as $R_{\p M}(p)$ contains also the information from long geodesics. We refer to \cite{Katchalov2001} for a thorough review of the related literature.

Also  the case of partial data has been considered. Let $\mathcal S, \mathcal R \subset \partial M$ be open and nonempty. In this case the inverse problem is the following. Does the restriction operator 
$$
\Lambda_{g,\S,\cR}f:=\Lambda_{g}f|_{\R_+ \times \cR}, \: f\in C^\infty_0(\R_+ \times \cS)
$$  
with $\cS,\cR$ determine $(\overline M,g)$? The answer is positive if $\cS=\cR$, (see \cite{Katchalov1998}). In \cite{lassas2010inverse} the problem has been solved in the case $\overline{\cR}\cap \overline{\cS}\neq \emptyset$. The general case is still an open problem.  So far the sharpest results are \cite{lassas2014inverse, milne2016codomain}.    


\subsubsection{Distance difference functions}

In \cite{LaSa} the collection of \textit{distance difference functions} 
\begin{equation}
\label{eq:DDD}
\{(d_g(p,\cdot)-d_g(p,\cdot))|_{N \setminus M}; p\in M\}\subset L^\infty(N \setminus M),
\end{equation}
was studied. The authors show that $(N \setminus M, g|_{N \setminus M})$ with the collection \eqref{eq:DDD} determines the Riemannian manifold $(N,g)$ up to an isometry, if $N$ is a compact smooth manifold of dimension two or higher and $M$ is open and has a smooth boundary. In the context of this paper one can assume that at the unknown area $M$ there occurs an Earthquake at an unknown point $p\in M$ at an unknown time $t\geq 0$. This Earthquake emits a seismic wave. At the every point of the measurement area $N \setminus M$ there is a device that records the time when the seismic wave hits the corresponding point. This way we obtain the function $(z,w)\mapsto D_p(z,w):=d_g(p,z)-d_g(p,w), \: z,w \in N \setminus M$ that is the travel time difference of the seismic wave.

Fix $z \in \p M$. Suppose that for every $p \in M$ the corresponding distance difference function $D_p(\cdot,z):\p M \to \R$ is smooth. Then there is a close connection between $R_{\p M}(p)$ and
$$
R_{\p M}^1(p):=\bigg\{\nabla_{\p M}D_p(\cdot,z)\bigg|_{q};q\in \p M\bigg\},
$$
Notice that there exists non-trapping manifolds $\overline{M}$ with strictly convex boundary such that for a given $p \in M$ the mapping $\p M\ni q \mapsto d_g(q,p)$ is not smooth.

Moreover to prove the Main theorem \ref{th:main} of this paper we will use techniques that where introduced in \cite{LaSa}. These techniques are based on the Theorem 1 of \cite{topalov2003geodesic}.
\subsubsection{Spherical surface data} Finally we will present one more geometric inverse problem where a geodesical measurement data is considered.  Let $(N,g)$ be a complete or closed Riemannian manifold of dimension $n\in \N $ and $M\subset N$ be an open subset of $N$ with smooth boundary. We denote by $U:=N\setminus \overline{M}$. 

%

In \cite{deHoop1} one considers the \textit{Spherical surface} data 
consisting of the set $U$  and the collection of all pairs $(\Sigma,r)$ where $\Sigma\subset U$ is a smooth $(n-1)$ dimensional submanifold
that can be written in the form
$$
\Sigma=\Sigma_{x,r,W}= \{\exp_x(rv) \in N; v\in W\},
$$ 
where $x\in M$, $r>0$ and $W\subset S_xN$ is an open and connected set. Such surfaces $\Sigma$ are called spherical surfaces, or more precisely, subsets of generalised spheres of radius $r$. We point out that the unit normal vector field 
$$
\nu_{x,r,W}(\gamma_{x,v}(r)):=\dot\gamma_{x,v}(r), \quad v \in W
$$
of the generalized sphere $\Sigma_{x,r,W}$ determines a data that has a natural connection to our data $R_{\p M}(M)$. 

Also, in \cite{deHoop1} one assumes that $U$ is given with its $C^\infty$-smooth coordinate atlas. 
Notice that in general, the spherical surface $\Sigma$ may be related to many centre points and radii. For instance consider the case where $N$ is a two dimensional sphere.

In \cite{deHoop1} it is shown that the Spherical surface data determine uniquely the Riemannian structure of $U$. However these data are not sufficient to determine $(N,g)$ uniquely. In \cite{deHoop1} a counterexample is provided. In \cite{deHoop1} it is shown that the Spherical surface data  determines the universal covering space of $(N,g)$ up to an isometry. 

In \cite{de2014reconstruction} a special case of problem \cite{deHoop1} is considered. The authors study a setup where $M \subset \R^n, \: n\geq 2$ and the metric tensor $g|_{\overline{M}}=v^{-2}e$, for some smooth and strictly positive function $v$. Let $x \in \R^n \setminus M$ and $y :=\gamma_{x,\xi}(t_0) \in M$, for some $\xi \in S_xN$ and $t_0>0$, such that $y$ is not a conjugate point to $x$ along $\gamma_{p,\xi}$. The main theorem of \cite{de2014reconstruction} is that, if the coefficient functions of the shape operators of generalized spheres $\Sigma_{y,r,W}$ are known in a neighborhood $V$ of $\gamma_{x,\xi}([0,t_0))$ and the wave speed $v$ is known in $(N \setminus M)\cap V$, then the wave speed $v$ can be determined in some neighborhood $V' \subset V$ of $\gamma_{x,\xi}([0,t_0])$.


\section{Strictly convex manifolds, about the  extension of the data and the generic property}  
\subsection{Analysis of strictly convex boundaries}
We will write $(z(p),s(p))$ for the boundary normal coordinates of $\p M$, for a point $p$ near $\p M$. (See 2.1,  \cite{Katchalov2001}, for the definitions). Here for any $p \in dom((z,s))$, the mapping $p\mapsto z(p)$ stands for the closest point $z(p)\in \p M$ of $p$ and $p\mapsto s(p)$ is the signed distance to the boundary, i.e.,
$$
|s(p)|=\dist_g(p,\p M)
$$
and
$$
s(p)<0 \hbox{ if } p \in M, \: s(p)=0 \hbox{ if } p \in \p M \hbox{ and } s(p)>0 \hbox{ if } p \in N \setminus \overline{M}.
$$
Thus every $q \in \p M$ has a neighborhood $U \subset N$, such that the map $U \ni p \mapsto (z(p),s(p))$ is a smooth local coordinate system. We will write $\nu$ for the unit outer normal of $\p M$, therefore, $\nabla_g s(p)=\nu$ for $p\in \p M$. Therefore the distance minimizing geodesic from $p$ to $z(p)$ is normal to $\p M$.

\begin{Definition}
The second fundamental form of $\p M$ is said to be positive definite, i.e., the boundary $\p M$ as a submanifold is strictly convex, if the corresponding shape operator 
\begin{equation}
\label{eq:shape_operator}
S:T\p M \to T\p M, \: S(X)=\nabla_X \nu
\end{equation}
is positive definite. Here $\nabla$ is the Riemannian connection of the metric tensor $g$.
\end{Definition} 
Next we recall a few well known results (without proof) about manifolds with strictly convex boundary.
\begin{Lemma}
\label{LE:2_fundamental_form}
Let $(N,g)$ be a complete smooth Riemannian manifold and $M\subset N$ open subset with smooth boundary. Suppose that the second fundamental form of $\p M$ is positive definite. Then for each $(p,\xi) \in S\p M$ there exists $\epsilon>0$ such that the geodesic $\gamma_{p,\xi}(t)\notin \overline M$ for any $t\in (0,\epsilon)$.

Moreover if $(p,\xi) \in S\overline{M} \setminus S\p M$, then the maximal geodesic $\gamma_{p,\xi}:[a,b] \to \overline{M}$ does not hit the boundary $\p M$ tangentially. 

\end{Lemma}

\begin{Corollary}
\label{Co:interior_of_geo}
Let $(N,g)$ be a complete smooth Riemannian manifold and $M\subset N$ open subset with smooth boundary. Suppose that the second fundamental form of $\p M$ is positive definite. Let $(p,\xi) \in S \overline{M}$. Then
$$
\gamma_{p,\xi}((0,\tau_{exit}(p,\xi)))\subset M.
$$
\end{Corollary}
\begin{proof}
The claim follows from Lemma \ref{LE:2_fundamental_form} and \eqref{eq:exit_time_function}.
\end{proof}

\begin{Lemma}
\label{Le:minimizing_geo_from_boundary}
If $\p M$ is strictly convex, then for any $p \in \overline{M}$ there exists a neighborhood $U \subset N$ of $p$ such that for all $z,q \in U\cap \overline{M}$ the unique distance minimizing unit speed geodesic $\gamma$ from $z$ to $q$ is contained in $U$ and $\gamma(t)\in M \cap U$ for all $t \in (0,d(z,q))$. 
\end{Lemma}

\begin{Lemma}
\label{Le:exit_time_func_is_cont}
If  $(\overline{M},g)$ is compact, non-trapping and $\p M$ is strictly convex, then the first exit time function $\tau_{exit}:S\overline{M}\to \R_+$ given by \eqref{eq:exit_time_function} is continuous and there exists $L>0$ such that 
\begin{equation}
\label{eq:exit_time_func_is_bdd}
\max\{\tau_{exit}(p,\xi);\: (p,\xi)\in S\overline{M}\}\leq L.
\end{equation}
Moreover, the exit time function $\tau_{exit}$ is smooth in $S \overline M \setminus S\p M$ and in $\p_+ SM$.
\end{Lemma}
\begin{proof}
See Section 3.2. of \cite{sharafutdinov1999ray}.
\end{proof}
\subsection{Extension of the measurement data}
We start with showing that the data \eqref{geodesic measurement data1} determines the boundary metric. This is formulated more precisely in the following lemma.
\begin{Lemma}
\label{Le:complete:scattering_data}
Let $(N,g)$ be a smooth, closed, connected Riemannian $n$-dimensional manifold $n\geq 2$, $M\subset N$ an open set with smooth boundary. Then \eqref{geodesic measurement data1} determines the data
\begin{equation}
((\p M, g|_{\p M}) ,R_{\p M}(M)),
\label{geodesic measurement data}
\end{equation}
where $g|_{\p M}= i^{\ast}g $ and $i:\p M \hookrightarrow \overline M$. 

More precisely if $(N_i,g_i)$ and $M_i$ are as in Theorem \ref{th:main} and  \eqref{eq:collar_neib1}--\eqref{eq:equivalenc_of_data} hold, then
\begin{equation}
\label{eq:collar_neib}
\hbox{$\phi: \p M_1 \rightarrow \p M_2$ is a diffeomoprhism  such that $\phi^\ast (g_2|_{\p M_2})=g_1|_{\p M_1}$.}
\end{equation}

\end{Lemma}
\begin{proof}
We will start with showing that we can recover the metric tensor $g|_{\p M}$. Choose $p \in \p M$, $\eta \in (T_p\p M\setminus \{0\})$ and consider the set
\begin{equation}
\label{eq:tangential_comp_of_exit_directions}
\begin{array}{lll}
R(\eta)&:=&\{\xi \in T_p \p M; \hbox{ there exists  $z \in M$ such that } \xi \in R_{\p M}(z)
\\
&& \hbox{and } a>0, \hbox{ that satisfies }\: \xi=a\eta \}.
\end{array}
\end{equation}
It is easy to see that the set $R(\eta)$ is not empty and by Lemma \ref{LE:2_fundamental_form} we also have
\begin{equation}
\label{eq:boundary_metric1}
\{\|\xi\|_g \in \R_+; \xi \in R(\eta)\}=(0,1) \subset \R.
\end{equation}
Therefore, 
\begin{equation}
\label{eq:boundary_metric2}
\|\eta\|_g=\inf\{a>0;\frac{\eta}{a} \in R(\eta)\}.
\end{equation}
Since $\eta \in T_p\p M$ was arbitrary we have recovered the $g$-norm function $\|\cdot\|_g:T_p \p M \to \R$. Since norm $\|\cdot\|_g$ is given by the $g$-inner product,  we recover $\langle \cdot, \cdot \rangle_g$ using the parallelogram rule, that is
$$
\langle \xi, \eta\rangle_g=\frac{\|\xi\|_g^2+\|\eta\|_g^2-\|\xi-\eta\|_g^2}{2}, \quad \eta, \xi \in T_p \p M.
$$

\medskip 
Suppose then that for manifolds $(\overline{M_1},g_1)$ and $(\overline{M_2},g_2)$ properties \eqref{eq:collar_neib1}--\eqref{eq:equivalenc_of_data} are valid. Choose $p \in \p M_1$ and $\eta \in T_p \p M_1$ and denote $R(\eta)$ as in \eqref{eq:tangential_comp_of_exit_directions}. Then by \eqref{eq:collar_neib1}--\eqref{eq:equivalenc_of_data} we have
$$
D\phi (R(\eta))=R(D \phi \eta)\subset T_{\phi(p)}\p M_2.
$$
and by the equations \eqref{eq:boundary_metric1}--\eqref{eq:boundary_metric2} we have $\|\eta\|_{g_1}=\|D \phi \eta\|_{g_2}$  and equation \eqref{eq:collar_neib} follows from the parallelogram rule. 
\end{proof}

Let $p \in \overline M$, we define \textit{the complete scattering set of the point source} $p$ as 
\begin{equation} 
\label{eq:full_point_source_set}
\begin{array}{lll}
R^E_{\p M}(p)&=& \{(q,\eta)\in \p_- SM; \: \textrm{ there is }\xi\in S_pN 
\\
&& \textrm{ and } t\in [0,\tau_{exit}(p,\xi)] \textrm{ such
that }\\
&& q=\gamma_{p,\xi}(t), \eta=\dot  \gamma_{p,\xi}(t)\}.
\end{array}
\end{equation}
We emphasize that the difference between $R_{\p M}(p)$ and $R_{\p M}^E(p)$ for $p \in \overline M$ is that $R_{\p M}(p)$ contains the tangential components of vectors $\xi \in R^E_{\p M}(p)$. We denote 
$R^E_{\p M}(M):=\{R^E_{\p M}(p); p \in M\}$. In the next Lemma we will give an equivalent definition for \eqref{eq:full_point_source_set}.
\begin{Corollary}[Corollary of Lemma \ref{LE:2_fundamental_form}]
\label{Co:new_def_for_data}
For any $p \in \overline{M}$
$$
R^E_{\p M}(p)=\{(\gamma_{p,\xi}(\tau_{exit}(p,\xi)),\dot{\gamma}_{p,\xi}(\tau_{exit}(p,\xi))); \: \xi \in S_pN\}.
$$
\end{Corollary}
\begin{proof}
Since $\p M$ is strictly convex and $M$ is non-trapping, the claim follows from the definition of $R^E_{\p M}(p)$ and Corollary \ref{Co:interior_of_geo}.
\end{proof}

Let $(N_i,g_i)$ and $M_i$ be as in Theorem \ref{th:main} and suppose that \eqref{eq:collar_neib1}--\eqref{eq:equivalenc_of_data} hold. Since $\overline{M}_i$ is compact, there exists $r>0$ and a neighborhood $K_i \subset \overline{M}_i$ of $\p M_i$ such that the mapping 
$$
f_i:\p M_i \times (-r,0]\to K_i \subset \overline{M}_i, \: f_i(z,s):=\exp_z(s\nu(z))
$$
is a diffeomorphism. 
By \eqref{eq:collar_neib1} the map
$$
\widetilde \Phi:\p M_1 \times (-r,0]\to \p M_2 \times (-r,0],\: \widetilde \Phi(z,s)=(\phi(z),s).
$$
is a diffeomorphism. Thus the mapping
\begin{equation}
\label{eq:big_phi}
\Phi:K_1 \to K_2, \: \Phi=f_2 \circ \widetilde \Phi \circ f_1^{-1},
\end{equation}
is a diffeomorphism between $K_1$ and $K_2$. Moreover the map $\widetilde \Phi$ is a local representation of $\Phi$ in the boundary normal coordinates.

\begin{Lemma}
\label{Le:extension_from_tangential_data_to_full_data}
Suppose that $(N,g)$ and $M$ are as in Theorem \ref{th:main}. Then the data \eqref{geodesic measurement data} determines the complete scattering data 
\begin{equation}
((\p M, g|_{\p M}), R^E_{\p M}(\overline{M})).
\label{bigger geodesic measurment data}
\end{equation}

More precisely if $(N_i,g_i)$ and $M_i$ are as in Theorem \ref{th:main} such that \eqref{eq:equivalenc_of_data} and \eqref{eq:collar_neib} hold, then 
\begin{equation}
\label{eq:Phi_preserves_boundary_inner_product}
\langle \eta, \xi \rangle_{g_1}=\langle D \Phi \eta, D\Phi\xi \rangle_{g_2}, \hbox{ for all } \xi,\eta \in \p TM_1
\end{equation}
and
\begin{equation}
\label{eq:equivalence_of_biger_data}
\{D\Phi (R^E_{\p M_1}(q));\ q\in \overline{M}_1\}= \{ R^E_{\p M_2}(p);\ p\in \overline{M}_2\}.
\end{equation} 
\end{Lemma}
\begin{proof}
We will start with showing that for every $q \in M$ we can recover $R_{\p M}^E(q)$.
Choose $q \in M$ and let $(p,\eta)\in \p_- S\overline{M}$ be such that $(p,\eta^T) \in R_{\p M}(q)$. Then $\|\eta^T\|_g < 1$ and in the boundary normal coordinates $x \mapsto (z(x),s(x))$ close to $p$  the vector $\eta$ can be written as
\begin{equation}
\label{eq:rep_of_boundary_vector}
\eta=\eta^T+\big(\sqrt{1-\|\eta^T\|_g^2}\big)\nu.
\end{equation}
Since $g|_{\p M}$ is known we have recovered $\eta$ in the boundary normal coordinates. Since $q\in M$ was arbitrary, we have recovered the collection $R^E_{\p M}(M)$.

Next we will show that for every $q \in \p M$ we can recover $R_{\p M}^E(q)$. Let $q \in \p M$ and $(p,\eta)\in \p_- SM$. We will define a set
\begin{equation}
\label{eq:image_of_open_geo_1}
\begin{array}{lll}
\Sigma(p,\eta)&:=&\{R^E_{\p M}(q') \in R_{\p M}^E(M); \:  (p,\eta) \in R^E_{\p M}(q')\}
\\
&=&\{R^E_{\p M}(q')\in R_{\p M}^E(M): q' \in \gamma_{p,-\eta}(0,\tau_{exit}(p,-\eta))\},
\end{array}
\end{equation}
with which we can verify if $(p,\eta)\in R^E_{\p M}(q)$ or not. Notice that by Lemma \ref{LE:2_fundamental_form} $\Sigma(p,\eta)=\emptyset$ if and only if $\gamma_{p, -\eta}([0,\tau_{exit}(p,-\eta)])\cap M$ is empty if and only if $\eta$ is tangential to the boundary. 

Suppose first that $\Sigma(p,\eta) \neq \emptyset$. Then $(p,\eta)\in R^E_{\p M}(q)$ if and only if there exists $(q,\xi)\in S_qN$ such that $\Sigma(p,\eta)=\Sigma(q,\xi)$. Suppose then that $\Sigma(p,\eta)=\emptyset$ and thus $\eta$ is tangential to $\p M$. By Lemma \ref{LE:2_fundamental_form} holds $\tau_{exit}(p,\eta)=0$. Therefore, $\eta \in R_{\p M}^E(q)$ if and only if $p=q$. We conclude that we can always check with data \eqref{geodesic measurement data} if $(p,\eta)\in \p_- SM$ is included in $R^E_{\p M}(q)$ or not. Since $q\in \p M$ was arbitrary, we have recovered the collection $R^E_{\p M}(\p M)$. Therefore, we have recovered the set $R_{\p M}^E(\overline{M})$.

\medskip
Next we will verify \eqref{eq:Phi_preserves_boundary_inner_product} and \eqref{eq:equivalence_of_biger_data}. Let $(N_i,g_i)$ and $M_i$ be as in Theorem \ref{th:main} such that \eqref{eq:equivalenc_of_data} and \eqref{eq:collar_neib} hold.  Choose $p \in \p M_1$ and $\xi, \eta \in T_p \p M_1$. By \eqref{eq:collar_neib}, \eqref{eq:big_phi} and \eqref{eq:rep_of_boundary_vector} we have
$$
\langle D\Phi \xi, D\Phi \eta  \rangle_{g_2}=\langle D\phi \xi^T, D\phi \eta^T  \rangle_{g_2}+\sqrt{1-\|D\phi \eta^T\|_{g_2}^2}\sqrt{1-\|D \phi\xi^T\|_{g_2}^2}=\langle \xi,  \eta  \rangle_{g_1}.
$$
Thus the equation \eqref{eq:Phi_preserves_boundary_inner_product} is valid. 

Next we will prove
\begin{equation}
\label{eq:full_interiori_equivalence_data}
\{D\Phi (R^E_{\p M_1}(q));\ q\in M_1\}= \{ R^E_{\p M_2}(p);\ p\in M_2\}.
\end{equation} 
Let $q \in M_1$. By \eqref{eq:equivalenc_of_data} there exists $z \in M_2$ such that 
$$
R_{\p M_2}(z)=D\phi R_{\p M_1}(q).
$$ 
Let $(p,\eta) \in  R^E_{\p M_1}(q)$. Then $\|\eta\|_{g_1}=1$ and $ \eta^T\in R_{\p M_1}(q)$. Since $D\phi \eta^T\in R_{\p M_2}(z)$ we conclude by \eqref{eq:Phi_preserves_boundary_inner_product} that $D\Phi(p,\eta) \in R^E_{\p M_2}(z)$.
Therefore, we have $D\Phi (R^E_{\p M_1}(q)) \subset R^E_{\p M_2}(z) $. By symmetric argument we also have $R^E_{\p M_2}(z)\subset D\Phi( R^E_{\p M_1}(q))$. Thus we have proved the left hand side inclusion in  \eqref{eq:full_interiori_equivalence_data}. Again by symmetric argument we prove the right hand side inclusion in \eqref{eq:full_interiori_equivalence_data}.
 
Next we will show that
$$
\{D\Phi (R^E_{\p M_1}(q));\ q\in\p M_1\}= \{ R^E_{\p M_2}(p);\ p\in \p M_2\}.
$$
Let $q \in \p M_1$. By the definition of the mapping $\Phi$ and the equation \eqref{eq:Phi_preserves_boundary_inner_product} it is enough to prove that
\begin{equation} 
D\Phi (R^E_{\p M_1}(q))=R^E_{\p M_2}(\phi(q)).
\label{eq:equality_of_measurement_sets}
\end{equation}
Let $(p,\eta) \in  R^E_{\p M_1}(q)$ and denote $\Sigma(p,\eta)$ as in \eqref{eq:image_of_open_geo_1}.
By equations \eqref{eq:Phi_preserves_boundary_inner_product} and  \eqref{eq:full_interiori_equivalence_data} it holds that the set
\begin{equation}
\label{eq:geodesics_are_the_same}
\begin{array}{ll}
\Sigma(\Phi p, D\Phi \eta)& :=\{R_{\p M_2}(z);z\in M_2, D\Phi \eta \in R_{\p M_2}(z)\}
\\
& =\{D\Phi R_{\p M_1}(q');R_{\p M_1}(q')\in \Sigma(p,\eta)\}
\end{array}
\end{equation}
is empty if and only if $\Sigma(p,\eta)$ is empty. Suppose first that $\Sigma(p,\eta)$ is empty and thus $p=q$ and $\eta$ is tangential to the $\p M_1$. Therefore, by equation \eqref{eq:collar_neib} we have
$$
D\Phi \eta=D\phi \eta \in \bigg(R_{\p M_2}(\phi(q)) \cap R^E_{\p M_2}(\phi(q))\bigg).
$$ 
Suppose then that $\Sigma(p,\eta)$ is not empty. This implies that there exists a vector $\xi \in S_qN_1$ that satisfies $\Sigma(p,\eta)=\Sigma(q,\xi)$. Thus $\Sigma(\Phi p, D\Phi \eta)=\Sigma(\Phi q, D\Phi \xi)$ and this implies that $D\Phi(p,\eta) \in R_{\p M_2}(\phi(q))$. This completes the left hand inclusion of \eqref{eq:equality_of_measurement_sets}. Replace $\Phi$ by $\Phi^{-1}$ to prove the right hand side inclusion of \eqref{eq:equality_of_measurement_sets}. 

Therefore, \eqref{eq:equivalence_of_biger_data} is valid.
\end{proof}

\subsection{A generic property}
We will now formulate a generic property in $\Met(N)$ that is related to the complete scattering data \eqref{bigger geodesic measurment data}. 
\begin{Definition} \label{generic property}
Let $N$ be a smooth manifold and $M\subset N$ an open set. We say that a Riemannian metric $g \in \Met(N)$ separates the points of set $M$ if for all $p,q \in M, \: p\neq q$ there exists $\xi \in S_pN$ such that
$$
q \notin \gamma_{p,\xi}([-\tau_{exit}(p,-\xi),\tau_{exit}(p,\xi)]). 
$$
This means that there exists a geodesic segment that starts and ends at the boundary of $M$ and contains $p$, but does not contain  $q$.
\end{Definition}

We emphasize that there are easy examples for $N$, $M$ and $g$ such that $g$ does not separate the points of $M$. For instance let $M \subset S^2$ be a polar cap strictly larger than the half sphere. Consider any two antipodal points $p,q \in M$. Then the standard round metric does not separate the points $p$ and $q$.

Our next goal is to show that for every $(N,g)\in \mathcal{G},$ (where $\mathcal G$ is as in  \eqref{eq:admissible_metrics}) and  $M\subset N$  that is open, non-trapping in the sense of \eqref{eq_exit time is bounded} and has a smooth strictly convex boundary, the metric $g$ separates the points of $M$. This will be used in Section 3 to show that for two points $p,q \in M$ the complete scattering sets $R^E_{\p M}(p)$ and $R^E_{\p M}(q)$ coincide if and only if $p$ is $q$.
\begin{Lemma}
\label{Le:curve_xi(s)}
Let $(N,g)$ be a compact Riemannian manifold. Let $M\subset N$ be open, non-trapping in the sense of \eqref{eq_exit time is bounded} and have a smooth strictly convex boundary. Suppose that the metric $g$ does not separate the points of $M$. Then there exist $p, q\in M$, 
$p\neq q$, an interval $I\subset \R$, and a $C^\infty$-map $I \ni s \mapsto(\xi(s), \ell(s))\in S_pN\times \R,$ such that $\dot{\xi}(s)\neq 0$ and $q=\gamma_{p,\xi(s)}(\ell(s))$.
\end{Lemma}
\begin{proof}
Since $g$ does not separate the points of $M$ there are $p,q\in M$, $p\not =q$,  such that for all $\xi\in S_pN$
we have
\begin{equation}
\label{eq:times_when_hit_to_q}
S(\xi)=\{s\in (-\tau_{exit}(p,-\xi),\tau_{exit}(p,\xi));\ \exp_p(s\,\xi)=q \}\not =\emptyset.
\end{equation}
We note that for all $\xi \in S_pN$ the set $S(\xi)$ is finite due to inequality \eqref{eq:exit_time_func_is_bdd}. 

Fix $\eta \in S_pN$ and enumerate $S(\eta)=\{s_1,\ldots,s_K\}$ for some $K\in \N$. Then 
$$
-\tau_{exit}(p,-\eta)<s_1<s_2<, \ldots, <s_K<\tau_{exit}(p,\eta).
$$ 
For each $s_k\in \{s_1,\ldots,s_K\}$ we choose a $n-1$-dimensional submanifold $S_k$ of $N$ such that
$$
\gamma_{p,\eta}(s_k)=q\in S_k \hbox { and } \dot{\gamma}_{p,\eta}(s_k)\perp T_qS_k.
$$
For instance we can define
$$
S_k:=\{\exp_q(tv)\in N; v \in S_qN, \: v \perp \dot{\gamma}_{p,\eta}(s_k), \:  t \in [0,inj(q))\},
$$
where $inj(q)$ is the injectivity radius at $q$. Choose $\epsilon>0$ and consider a $\epsilon$-neighborhood $W$ of $\gamma_{p,\eta}([-\tau_{exit}(p,-\eta),\tau_{exit}(p,\eta)])$. We will write $
\widetilde S_k$ for the component of $ S_k \cap W $ that contains $q$. If $\epsilon$ is small enough, there exists $\delta, \delta'>0$ such that
for any 
$$
t \in [-\tau_{exit}(p,-\eta),\tau_{exit}(p,\eta)] \setminus \bigg(\bigcup_{k=1}^K(s_k-\delta,s_k+\delta)\bigg)
$$ 
holds 
$$
dist_g(\widetilde S_k,\gamma_{p,\eta}(t))\geq 2\delta', \hbox{ for every } k \in \{1,\ldots,K\}.
$$
By the continuity of the exponential mapping, we can choose a smaller $\epsilon>0$ such that there exists an open neighborhood $V\subset S_pN$ of $\eta$ such that for any 
$$
t \in [-\tau_{exit}(p,-\eta),\tau_{exit}(p,\eta)] \setminus \bigg(\bigcup_{k=1}^K(s_k-\delta,s_k+\delta)\bigg) \hbox{ and } \xi\in V
$$ 
holds 
\begin{equation}
\label{eq:too_faraway_times}
dist_g(\widetilde S_k,\gamma_{p,\xi}(t))\geq \delta', \hbox{ for every } k \in \{1,\ldots,K\}.
\end{equation}

Next we define a signed distance function 
$$
\varrho_k(t,\xi)=\left\lbrace\begin{array}{c}
-dist_g(\gamma_{p,\xi}(t),\widetilde S_k), (t,\xi) \in (s_k-\delta,s_k]\times V
\\
dist_g(\gamma_{p,\xi}(t),\widetilde S_k), (t,\xi) \in [s_k,s_k+\delta)\times V.
\end{array} \right.
$$
See Figure \eqref{fig:generic}. 

Choosing a smaller $\delta$ and $V$, if needed, it follows that the function 
$\varrho_k: (s_k-\delta,s_k+\delta)\times V \to \R$ is smooth. Then we have $\varrho_k(s_k,\eta)=0$ and
$$
\bigg|\frac{\p }{\p t}\varrho_k(t,\eta)|_{t=s_k}\bigg|=\|\dot{\gamma}_{p,\eta}(s_k)\|^2=1.
$$
Therefore, by the Implicit function theorem there exists an open neighborhood $V_k\subset V$ of $\eta$ and a smooth function $f_k:V_k\to (s_k-\delta,s_k+\delta)$ that solves the equation
\begin{equation}
\label{eq:implicit_func_th}
\varrho_k(f_k(\xi),\xi)=0.
\end{equation}
Define an open neighborhood $U$ of $\eta$ by $U:= \bigcap_{k=1}^K V_k$ and sets $U_k=\{\xi\in U: \exp_p(f_k(\xi)\xi)=q\}$. Then sets $U_k$ are closed in the relative topology of $U$.
By \eqref{eq:times_when_hit_to_q}, \eqref{eq:too_faraway_times} and \eqref{eq:implicit_func_th} it must hold that
\begin{equation}
\label{eq:U_k_deplete_U}
U=\bigcup_{k=1}^K U_k.
\end{equation}

We claim that for some $k \in \{1, \ldots, K\}$ it holds that $U_k^{int}\neq \emptyset$. If this is not true, then the sets $P_j:= U\setminus U_j, \: j\in\{1,\ldots, K\}$ are all open and dense in the relative topology of $U$. Moreover by \eqref{eq:U_k_deplete_U} we have
$$
P:=\bigcap_{j=1}^KP_j=\emptyset.
$$
This is a contradiction since $U$ is a locally compact Hausdorff space and thus by the Baire category theorem, it should hold that the set $P$ is dense in the relative topology of $U$. Thus there exists $k \in \{1, \ldots, K\}$ for which it holds that $U_k^{int}\neq \emptyset$. Choose an open $U' \subset U_k$. In particular $U'$ is open in $S_pN$ and there exists $\epsilon'>0$ and a $C^\infty$-path $\xi(s)$, $s\in (-\epsilon',\epsilon')$, in $U'$ such that $\dot \xi(s)\not =0$. Denoting   $\ell(s)= f_k(\xi(s))$, we have
 $q=\gamma_{p,\xi(s)}(\ell(s))$ for  $s\in (-\epsilon,\epsilon)$. Therefore, the curve $s\mapsto (\xi(s),\ell(s))$ satisfies the claim of this Lemma.
\end{proof}

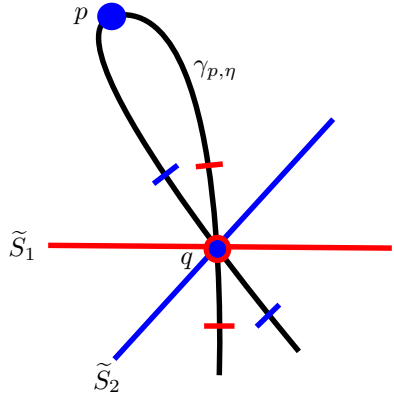
\begin{figure}[h]
 \begin{picture}(100,150)
  \put(-10,15){\includegraphics[width=130pt]{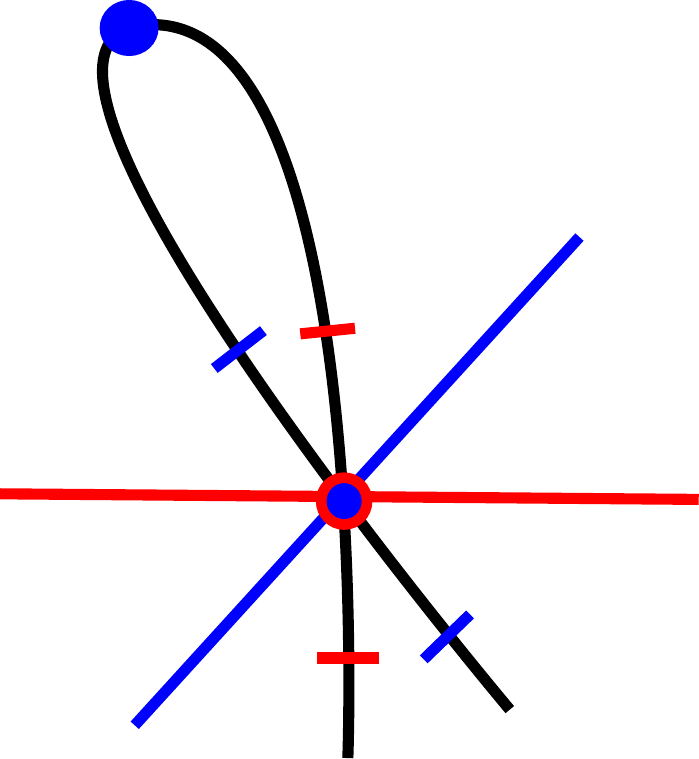}}
  \put(-25,60){$\widetilde S_1$}
  \put(0,150){$p$}
  \put(45,130){$\gamma_{p,\eta}$}
  \put(40,57){$q$}
  \put(7,10){$\widetilde S_2$}
   \end{picture}
\caption{Here is a visualization of the set up in the definition of the function $\varrho_k$ in Lemma \ref{Le:curve_xi(s)}. The blue dot is $p$ and the red\&blue dot is $q$. The black curve is the geodesic $\gamma_{p,\eta}$. The red line is the hypersurface $\widetilde S_1$ and the blue line is the hypersurface $\widetilde S_2$. The small blue and red segments indicate the intervals $(s_k-\delta, s_k+\delta)$ where the function $\varrho_k(\cdot,\eta)$ is defined.}
\label{fig:generic}
\end{figure}

\begin{Proposition}\label{generic A} 
Let $(N,g) \in \mathcal{G}$ and $M$ be as in Theorem \ref{th:main}. Then $g$ separates the points of $M$.
\end{Proposition}
\begin{proof}
We prove the proposition by contradiction. Given $g\in \mathcal{G}$, assume there are $p,q\in M, p\neq q$ such that for all $\xi\in S_pN$, 
$$
q \in \gamma_{p,\xi}([-\tau_{exit}(p,-\xi),\tau_{exit}(p,\xi)]). 
$$
In particular (see Lemma \ref{Le:curve_xi(s)}), there exists an open interval $(-\epsilon, \epsilon)$, such that for $s\in (-\epsilon,\epsilon)$, $(\xi(s),\ell(s))$ is a $C^1$-path on $S_pN\times \R$ such that $\dot{\xi}(s)\neq 0$ and $q=\gamma_{p,\xi(s)}(\ell(s))$. This implies 
\begin{equation}\label{variation}
\begin{array}{lll}
0 & =&\frac{\p}{\p s}(\gamma_{p,\xi(s)}(\ell(s)))\\
& =&\dot{\gamma}_{p,\xi(s)}(\ell(s))\cdot\frac{d}{ds}\ell(s)+\ell(s)\cdot D\,\mbox{exp}_p|_{\ell(s)\xi(s)}\dot{\xi}(s)\\
& =:&T_1+T_2.
\end{array}
\end{equation}
Since $\|\xi(s)\|=1$, it implies $\xi(s)\perp \dot{\xi}(s)$. By Gauss Lemma, we have $\dot{\gamma}_{p,\xi(s)}(\ell(s))\perp D\,\mbox{exp}_p|_{\ell(s)\xi(s)}\dot{\xi}(s)$, thus $T_1\perp T_2$. Applying equation  \eqref{variation}, we obtain $T_1=T_2=0$, therefore, $\frac{d}{ds}(\ell(s))=0$. This is equivalent to saying that $\ell(s)\equiv const,$ for $s\in (-\epsilon,\epsilon)$. This implies $I(g)=\infty$. By equation \eqref{eq:KPP_property}, we arrive a contradiction.
\end{proof}

\section{The reconstruction of the topology} 
We define a mapping 
\begin{equation*}
R^E_{\p M}: \overline{M}\rightarrow 2^{\p SM}, \: p \mapsto R^E_{\p M}(p).
\end{equation*}
The aim of this section is to show that the map $R^{E}_{\p M}$ is a homeomorphism, with respect to some suitable subset of $2^{\p SM}$ and topology of this subset. Recall that $(\p M, g|_{\p M})$ is known, by the Lemma \ref{Le:extension_from_tangential_data_to_full_data}. Therefore we can talk about topological properties of $2^{\p SM}$. Thus the goal is to reconstruct a homeomorphic copy of $(\overline{M}, g)$. We start with showing that the map $R^E_{\p M}$ is one-to-one, if the generic property \ref{generic property} holds.
\begin{Lemma} Suppose that $(N,g)$ and $M$ are as in the Theorem \ref{th:main}.
Then the mapping $R^E_{\p M}:\overline{M}\rightarrow 2^{\p SM}$ is one-to-one.
\label{R_F is one-to-one}
\end{Lemma}
\begin{proof}
We will prove the claim by contradiction. Therefore, we divide the proof into two separate cases.

Suppose first that there is $p\in \p M, \: q\in \overline{M}$ such that $R^E_{\p M}(p)=R^E_{\p M}(q)$, that is
\begin{eqnarray}
&&\{(\gamma_{p,\xi}(\tau_{exit}(p,\xi)),\dot{\gamma}_{p,\xi}(\tau_{exit}(p,\xi))); \: \xi \in S_pN\} \nonumber
\\
&=&\{(\gamma_{q,\eta}(\tau_{exit}(q,\eta)),\dot{\gamma}_{q,\eta}(\tau_{exit}(q,\eta))); \: \eta \in S_qN\} \subset \p_-SM. \nonumber
\end{eqnarray} 
By the definition of $R^E_{\p M}(p)$ it holds that 
$$
S_p \p M\subset R^E_{\p M}(p)=R^E_{\p M}(q).
$$
Therefore, we deduce that any tangential geodesic starting at $p$ hits $q$ before exiting $\overline{M}$. Since $\p M$ is strictly convex, this is true only if $p=q$.

Suppose then that there exist $p \in M,\: q \in \overline{M}, \:  p\neq q$ such that $R^E_{\p M}(p)=R^E_{\p M}(q)$. By the first part we may assume that $q \in M$. Let $\xi \in S_pN$. Then for $\pm \xi$ we have
$$
(\gamma_{p,\pm \xi}((\tau_{exit}(p,\pm\xi)),\dot \gamma_{p,\pm\xi}((\tau_{exit}(p,\pm\xi)))\in R^E_{\p M}(q).
$$
Therefore, it holds that
\begin{equation}
\label{eq:q_is_on_every_geo}
q \in \gamma_{p,\xi}((-\tau_{exit}(p,-\xi),\tau_{exit}(p,\xi))).
\end{equation}
Since $\xi \in S_pN$ was arbitrary, the equation \eqref{eq:q_is_on_every_geo} is valid for every $\xi \in S_pN$ and we have proved that the metric $q$ does not separate the points $p$ and $q$. This is a contradiction with Proposition \ref{generic A} and therefore, $p=q$.
\end{proof}

To reconstruct the topology of $\overline{M}$ from $R^E_{\p M}(\overline{M})$, we need to first give a topological structure to the latter. Notice that on the unit tangent bundle $S \overline{M}$ there is a natural metric induced by the underlying metric $g$, namely the Sasaki metric. We denote the Sasaki metric associated with $g$ by $g_S$. Then on the power set $2^{S \overline{M}}$, we assign the Hausdorff distance, i.e., given $A, B\in 2^{S\overline{M}}$
$$d_H(A,B):=\max\{\sup_{a\in A}\inf_{b\in B}d_{g_S}(a,b), \sup_{b\in B}\inf_{a\in A}d_{g_S}(a,b)\}.$$

However, in general the Hausdorff distance on $2^{S \overline{M}}$ needs not to be metrizable. If we consider the subset $\mathcal{C}(S \overline{M}):=\{\mbox{closed subsets of }\, \S \overline{M}\}\subset 2^{S \overline{M}}$, then $(\mathcal{C}(S \overline{M}), d_H)$ is a compact metric space by Blaschke selection theorem (see for instance \cite{ambrosio2004topics}, Theorem 4.4.15). The topology on $\mathcal{C}(S \overline{M})$ thus is induced by the Hausdorff metric $d_H$. 
We turn to consider a subspace $\mathcal{C}(\p SM):=\{\mbox{closed subsets of }\, \p S M\}\subset 2^{\p SM}$ of $\mathcal{C}(S\overline{M})$. Since the boundary $\p M$ is compact, $\mathcal{C}(\p S M)$ is a compact metric space. 

\begin{Lemma}
For any $q \in \overline{M}$, the set $R^E_{\p M}(q)\subset \p S M$ is compact.
\label{Le:measurement set is closed}
\end{Lemma} 
\begin{proof}
Consider a continuous mapping $E_q:S_qN \rightarrow \p_-S M$ given by
$$
E_q(\xi)=(\gamma_{q,\xi}(\tau_{exit}(q,\xi)),\dot \gamma_{q,\xi}(\tau_{exit}(q,\xi))).
$$
Since $S_qN$ is compact, and $E_q$ is continuous it holds by Corollary \ref{Co:new_def_for_data} that
$$
R^E_{\p M}(q)=E_q(S_qN) \hbox{ is compact}.
$$
\end{proof}

By Lemma \ref{Le:measurement set is closed} it holds that $R^E_{\p M}(\overline{M}) \subset \mathcal{C}(\p S M)$. From now on we consider the mapping 
\begin{equation}
R^E_{\p M}:\overline M \rightarrow  \mathcal{C}(\p S M), \: p \mapsto R^E_{\p M}(p).
\label{Mapping R_F}
\end{equation}

\begin{Definition}
Let $(X,d)$ be a complete metric space and $\mathcal C(X)$ be the collection of all closed subsets of $(X,d)$. We say that a sequence $(A_j)_{j=1}^\infty \subset \mathcal C(X)$ converges to $A \in \mathcal C(X)$ in the Kuratowski topology, if the following two conditions hold:
\begin{enumerate}
\item[(K1)] Given any sequence $(x_j)_{j=1}^\infty$, $x_j\in A_j$ with a convergent subsequence $x_{j_k}\to x$ as $k \to \infty$, then the limit point $x$ is contained in $A$.
\item[(K2)] Given any $x \in A$, there exists a sequence $(x_j)_{j=1}^\infty$, $x_j \in A_j$ that converges to $x$.
\end{enumerate}
\end{Definition}
\begin{Lemma}
\label{overline R_F is continuous}
The mapping $R^E_{\p M}:\overline{M} \rightarrow \mathcal{C}(\p S M)$ is continuous. 
\end{Lemma}
\begin{proof}
Let $q \in \overline{M}$ and $q_j \in \overline{M}, \: j\in \N$ be a sequence that converges to $q$. As $\p S M$ is compact it holds that the Kuratowski convergence and the Hausdorff convergence are equivalent (see e.g. \cite{ambrosio2004topics} Proposition 4.4.14). Thus it suffices to show that $R^E_{\p M}(q_j)$ converges to $R^E_{\p M}(q)$ in the space $\mathcal{C}(\p S M)$ in the sense of Kuratowski.

First we prove (K1). Let $(p_j,\eta_j) \in R^E_{\p M}(q_j)$. Changing in to subsequences, if necessary, we assume that $(p_j,\eta_j) \to (p,\eta) \in \p S M$ as $j \to \infty$. Let $\xi_j\in S_{q_j}N$ be such that for each $j\in \N$,
$$
\gamma_{q_j,\xi_j}(\tau_{exit}(q_j,\xi_j))=p_j \hbox{ and } \dot \gamma_{q_j,\xi_j}(\tau_{exit}(q_j,\xi_j))=\eta_j.
$$
As the first exit time function is continuous and $S \overline{M}$ is compact we may  with out loss of generality assume that $(q_j,\xi_j) \to (q',\xi) \in S_{q'}N, \: q' \in \overline{M}$ and  $\tau_{exit}(q_j,\xi_j) \to \tau_{exit}(q',\xi)$. Since $q_j \to q$, it holds that $q'=q$. 
By the continuity of the exponential mapping and the first exit time function the following holds
$$
(p,\eta)=\lim_{j \to \infty}(p_j,\eta_j)=\lim_{j \to \infty}(\gamma_{q_j,\xi_j}(\tau_{exit}(q_j,\xi_j)),\dot \gamma_{q_j,\xi_j}(\tau_{exit}(q_j,\xi_j)))
$$
$$
=(\gamma_{q,\xi}(\tau_{exit}(q,\xi)), \dot \gamma_{q,\xi}(\tau_{exit}(q,\xi))).
$$
Thus $(p,\eta)\in R^E_{\p M}(q)$. This proves (K1).

Next we prove (K2). Let $(p,\eta) \in R^E_{\p M}(q)$. Let $\xi\in S_{q}N$  be such that
$$
\gamma_{q,\xi}(\tau_{exit}(q,\xi))=p \hbox{ and } \dot \gamma_{q,\xi}(\tau_{exit}(q,\xi))=\eta.
$$
Since  $q_j \to q$ as $j\to \infty$ we can choose  $\xi_j \in S_{q_j}N,\: j \in \N$ such that $\xi_j \to \xi$ as $j \to \infty$. 
Denote
$$
\gamma_{q_j,\xi_j}(\tau_{exit}(q_j,\xi_j))=p_j \hbox{ and } \dot \gamma_{q_j,\xi_j}(\tau_{exit}(q_j,\xi_j))=\eta_j,
$$
since the exponential map and the first exit time function are continuous, then 
$$
(p_j,\eta_j) \in R^E_{\p M}(q_j) \hbox{ and } (p_j,\eta_j) \to (p,\eta).
$$  
This proves (K2).

We conclude that $R^E_{\p M}(q_j)$ converges to $R^E_{\p M}(q)$ in the Kuratowski topology. 
\end{proof}

\begin{Proposition}
\label{pr:R_F is homeo}
The mapping $R^E_{\p M}:\overline{M} \rightarrow R^E_{\p M}(\overline{M}) \subset \mathcal{C}(\p S M)$ is a homeomorphism.
\end{Proposition}
\begin{proof}
Note the $R^E_{\p M}:\overline{M} \rightarrow R^E_{\p M}(\overline{M}) \subset \mathcal{C}(\p S M)$ is continuous, one-to-one and onto, and $\overline M$ is compact and $\mathcal C(\p SM)$ is a metric space and thus a topological Hausdorff space. These yield that $R^E_{\p M}:\overline{M} \rightarrow R^E_{\p M}(\overline{M})$ is a homeomorphism.
\end{proof}

By the Proposition \ref{pr:R_F is homeo} the manifold topology of the data set $R^E_{\p M}(\overline{M})$ is determined. Thus the topological manifold $R^E_{\p M}(\overline{M})$  is a homeomorphic copy of $\overline{M}$. The rest of this section is devoted to constructing a map from $\overline{M}_1$ onto $\overline{M}_2$ that we later show to be a Riemannian isometry.
\color{black}

\medskip

Define a map 
\begin{equation}
\label{eq:hat_Phi}
\widehat{D\Phi}:\mathcal{C}(\p S M_1) \to \mathcal{C}(\p S M_2), \: F \mapsto D\Phi(F). 
\end{equation}

\begin{Lemma}
\label{Le:lift_of_Dphi_is_homeo}
The map $\widehat{D\Phi}:\mathcal{C}(\p S M_1) \to \mathcal{C}(\p S M_2)$ is a homeomorphism.
\end{Lemma}
\begin{proof}
We start by noticing that, if $(X,d_X)$ and $(Y,d_Y)$ are compact metric spaces, and $f:X \to Y$ is continuous, then the lift $\widehat{f}:\mathcal{C}(X) \to \mathcal{C}(Y), \: \widehat{f}(K)=f(K)$ is well defined. Let $(K_i)_{i=1}^\infty \subset \mathcal{C}(X)$ be a sequence that converges to $K \in \mathcal{C}(X)$ with respect to Kuratowski topology. We will show that $\widehat{f}(K_i)$ also converges to $\widehat{f}(K)$ with respect to Kuratowski topology, and by \cite{ambrosio2004topics} Proposition 4.4.14 this will imply that the lift $\widehat{f}$ is continuous.

Let $(y_i)_{i=1}^\infty, \; y_i \in f(K_i)$ be a sequence with a convergent subsequence $(y_{i_k})_{k=1}^\infty, \; y_{i_k}\in f(K_{i_k})$, we denote the limit point by $y\in Y$. Thus for each $i_k$, there exists $x_{i_k}\in K_{i_k}$ such that $f(x_{i_k})=y_{i_k}$. Since $X$ is a compact metric space the  sequence $(x_{i_k})_{k=1}^\infty$ has a convergent subsequence in $X$, we denote it by $(x_{i_k})_{k=1}^{\infty}$ again. 
Notice that $K_i\to K$ in the sense of Kuratowski convergence, we get $x_{i_k} \to x \in K$. By the continuity of $f$, one has 
$$
y=\lim_{k \to \infty}y_{i_k}=\lim_{k \to \infty}f(x_{i_k})=f(x)\in f(K).
$$

Let $y \in f(K)$ and $x \in K$ such that $f(x)=y$. Since $K_i \to K$ in the sense of Kuratowski, there exists a sequence $(x_i)_{i=1}^\infty, \: x_i \in K_i$ that converges to $x$. On the other hand, $f$ is continuous, thus the sequence $f(x_i)\in f(K_i)$ converges to $f(x)=y$. This completes the proof of the convergence of $f(K_i)$ to $f(K)$.

\medskip
Notice that by \eqref{eq:big_phi} the mapping $D \Phi:\p TM_1 \to \p TM_2$ is a smooth invertible bundle map. In particular $D \Phi$ is continuous with a continuous inverse. Therefore, $\widehat{D \Phi}$ is a well defined homeomorphism according to the first part of the proof.  
\end{proof}

Next we define a mapping
\begin{equation}
\label{eq:the_map}
\Psi: \overline{M_1} \rightarrow \overline{M}_2, \: \Psi:= (R_{\p M_2}^E)^{-1}\circ \widehat{D\Phi} \circ R^E_{\p M_1}.
\end{equation}
By \eqref{eq:equivalence_of_biger_data} and Proposition \ref{pr:R_F is homeo} the map $\Psi$ is well defined. Now we prove the main theorem of this section.
\begin{Theorem}
\label{th:topology}
Let $(N_i,g_i)$ and $M_i$ be as in Theorem \ref{th:main} such that \eqref{eq:collar_neib} and \eqref{eq:equivalence_of_biger_data} hold. Then the map $\Psi: \overline{M_1} \rightarrow \overline{M}_2,$ is a homeomorphism such that $\Psi|_{M_1}:M_1 \to M_2$ and $\Psi|_{\p M_1}=\phi.$
\end{Theorem}
\begin{proof}
By \eqref{eq:equivalence_of_biger_data}, Proposition \ref{pr:R_F is homeo} and Lemma \ref{Le:lift_of_Dphi_is_homeo} the map $\Psi$ is a  homeomorphism.

Let $p \in M_1$. Suppose that $q:=\Psi(p)\in \p M_2$, it holds that $S_q \p M_2\subset R^E_{\p M_2}(q)$.
On the other hand by the proof of Lemma \ref{R_F is one-to-one} there is no $p' \in \p M_1$ such that $S_{p'} \p M_1 \subset R^E_{\p M_1}(p)$. Since $D \Phi$ is an isomorphism, we reach a contradiction and thus $q \in M_2$. Similar argument for the inverse mapping $\Psi^{-1}$ and $p \in M_2$ proves the second claim.

Let $p \in \p M_1$, then $S_p \p M_1 \subset R^E_{\p M_1}(p)$  and by \eqref{eq:collar_neib}
$$
D\Phi (S_p \p M_1)= S_{\phi(p)} \p M_2 \subset R^E _{\p M_2}(\phi(p)).
$$
Therefore, by the proof of Lemma \ref{R_F is one-to-one} we have $\Psi(p)=\phi(p)$.
\end{proof}

\section{The reconstruction of the differentiable structure}

In this section we will show that the map $\Psi:\overline{M}_1\to \overline{M}_2$ is a diffeomorphism. First we will introduce a suitable coordinate system of smooth manifold $\overline M_i$ that is compatible with the data \eqref{bigger geodesic measurment data}. We will consider separately coordinate charts for interior and boundary points. Then we will define a smooth structure on topological manifold $R^E_{\p M_i}(\overline M_i)$ such that the maps $R^E_{\p M_i}: \overline{M}_i \to R^E_{\p M_i}(\overline{M}_i)$ are diffeomorphisms. The third task is to show that  the map $\widehat{D \Phi}:R^E_{\p M_1}(\overline{M}_1) \to R^E_{\p M_2}(\overline{M}_2)$ is a diffeomorphism with respect to smooth structures of $R^E_{\p M_i}(\overline M_i)$. By the formula \eqref{eq:the_map}, that defines the map $\Psi$, the following diagram commutes 
\begin{equation}
\label{eq:diagram}
\begin{array}{ccc}
\overline M_1 & \stackrel{R^E_{\p M_1}}{\longrightarrow} & R^E_{\p M_1}(\overline M_1)
\\
&&
\\
\bigg\downarrow \Psi & & \bigg\downarrow \widehat{D\Phi}
\\
&&
\\
\overline M_2 & \stackrel{R^E_{\p M_2}}{\longrightarrow} & R^E_{\p M_2}(\overline M_2),
\end{array}
\end{equation}
and the above steps prove that the map $\Psi$ is a diffeomorphism. Also we will  explicitly construct the smooth structure for $R^E_{\p M}(\overline M)$ only using the data \eqref{bigger geodesic measurment data}. In the steps below, we will often consider only one manifold and do not use the sub-indexes, $M_1$ and $M_2$, when ever it is not necessary. 

\subsection{The recovery of self intersecting geodesics and conjugate points}
\medskip
Let us choose a point $p\in \overline M$ for the rest of this section. The first part of the section is dedicated to finding for the point $p$ a suitable $q\in \p M$, a neighborhood $V_q$ of $p$ and to construct a map
\begin{equation}
\label{eq:n-1,coordinate_map}
\Theta_q(z):=\frac 1{\|\exp_q^{-1}(z)\|_g} \exp_q^{-1}(z)\in S_qN ,\quad z\in V_q .
\end{equation}
using our data \eqref{bigger geodesic measurment data}. 
We start with considering what does it require from $q \in \p M$ to be "suitable``.

\color{black}
\medskip
Let $(q,\eta)\in R^E_{\p M}(p)$, we say that $\eta$ is a conjugate direction with respect to $p$, if $p$ is a conjugate to $q$ on the geodesic $\gamma_{q,-\eta}$. This is equivalent to the existence  $s\in (0,\tau_{exit}(q,-\eta))$ such that $\gamma_{q,-\eta}(s)=p$ and a non-trivial Jacobi field $J$ such that $J(0)=0$ and $J(s)=0$. 

We say that a geodesic $\gamma$ is self-intersecting at the point $q \in \overline M$, if there exists $t_1<t_2$ such that 
$$
\gamma(t_1)=q=\gamma(t_2).
$$
Notice that the geodesic $\gamma_{q,-\eta}([0,\tau_{exit}(q,-\eta)])$ may be self-intersecting at $p$, and moreover there may be several $s'\in  (0,\tau_{exit}(q,-\eta))$ such that $\gamma_{q,-\eta}(s')=p$ and non-trivial Jacobi fields that vanish at $0$ and $s'$.

In the next Lemma we show that under the assumptions of Theorem \ref{th:main} most of the geodesics that start and end at $\p M$ and hit the point $p$, do it only one time.

\begin{Lemma}
\label{Le:non_intersection_directions}
Let $(N,g)$ and $M$ be as in the Theorem \ref{th:main}. Let $p \in \overline{M}$. Denote by
\begin{equation}
\label{eq:self_intersection_directions}
\begin{array}{lll}
I&:=&\{\xi\in S_pN; \hbox{ the geodesic } \gamma_{p,\xi}:[-\tau_{exit}(p,-\xi),\tau_{exit}(p,\xi)] \to \overline{M}
\\
&& \hbox{ is self-intersecting at } p\}.
\end{array}
\end{equation}
Then the set $S_pN \setminus I$ is open and dense. 

\medskip
If $r \in R_{\p M}^E(\overline M)$ is given and $p\in \overline{M}$ is the unique point for which $R_{\p M}^E(p)=r$ we define
\begin{equation}
\label{eq:allmost_good_directions}
K(p):=\{(q,\eta) \in (R^E_{\p M} (p) \setminus S_pN) ; \:  \gamma_{q,-\eta} \hbox{ is not self-intersecting at } p\}\subset \partial SM.
\end{equation}
See Figure \eqref{fig:K(p)}. 

Then $K(p)$ is not empty and $K(p)$ is determined by $r$ and data \eqref{bigger geodesic measurment data}.

More precisely if $(N_i,g_i)$ and $M_i$ are as in Theorem \ref{th:main} such that \eqref{eq:collar_neib} and \eqref{eq:equivalence_of_biger_data} hold, then $ \hbox{ for every } p \in \overline{M_1}$ the set $K(p)\neq \emptyset$ and 
\begin{equation}
\label{eq:almost_good_dir_are_equiv}
D\Phi (K(p))=K(\Psi(p)).
\end{equation}
\end{Lemma}

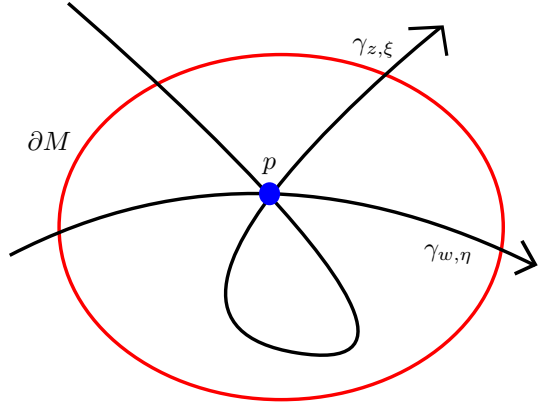
\begin{figure}[h]
 \begin{picture}(100,150)
  \put(-60,-10){\includegraphics[width=200pt]{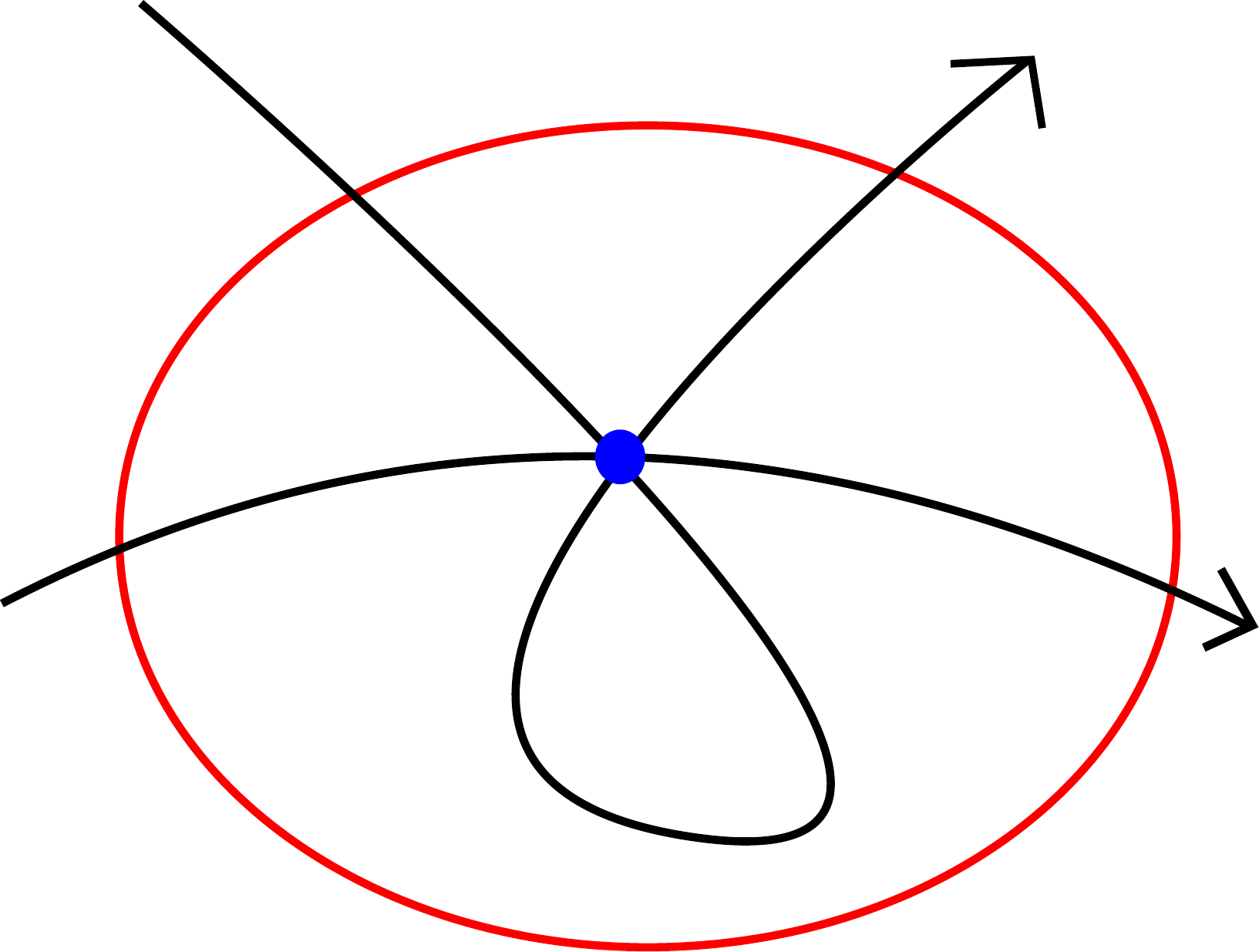}}
  \put(70,123){$\gamma_{z,\xi}$}
  \put(97,45){$\gamma_{w,\eta}$}
  \put(-53,85){$\p M$}
  \put(36,78){$p$}
   \end{picture}
\caption{Here is a schematic picture about $K(p)$, where the point $p \in M$ is the blue dot. The black curves represent the geodesics $\gamma_{z,\xi}$, and $\gamma_{w,\eta}$ respectively, where vectors $(z,\xi), (w,\eta) \in R_{\p M}^E(p)$. Notice that only $(w,\eta)\in K(p)$.}
\label{fig:K(p)}
\end{figure}

\begin{proof}
Suppose first that $p \in M$. We start with proving that $I$ is closed. Let $\xi \in \overline{I}$ and choose a sequence $(\xi_j)_{j=1}^\infty \subset I$ that converges to $\xi$ in $S_p N$. Choose a sequence $(t_j)_{j=1}^\infty \subset \R$ that satisfies
$$
\exp_p(t_j\xi_j)=p \hbox{ and } 0<|t_j|.
$$
By \eqref{eq:exit_time_func_is_bdd} we can without loss of generality assume that
$$
t_j \longrightarrow t \hbox{ as } j \longrightarrow \infty.
$$
Then it must hold that $0< |t|$ since any geodesic starting at $p$ cannot self-intersect at $p$ before time inj$(p)>0$, where inj$(p)$ is the injectivity radius at $p$. Therefore, $\exp_p(t\xi)=p$ and we have proved that $\xi \in I$. Thus $I$ is closed. 

From the proofs of Lemma \ref{Le:curve_xi(s)} and Proposition \ref{generic A} it follows that the set $I$ is nowhere dense, this is that $I$ does not contain any open sets. Thus $S_pN \setminus I$ is open and dense. Moreover the set
$$
K(p)=\{(\gamma_{p,\xi}(\tau_{exit}(p,\xi)),\dot \gamma_{p,\xi}(\tau_{exit}(p,\xi)))\in \p_-SM;  \: \xi \in S_pN \setminus I\}
$$ 
is not empty. 

Denote $r = R_{\p M}^E(p)$. Next we will show that $r$ and data \eqref{bigger geodesic measurment data} determine the set $K(p)$. Let $(q,\eta) \in r$. We write $J(q,\eta)$ for the set of self-intersection points of geodesic segment $\gamma_{q,-\eta}:(0,\tau_{exit}(q,-\eta)) \to M$. Due to non-trapping assumption of $M$ the set $J(q,\eta)$ is finite. Since the map $R_{\p M}^E: \overline{M}\to R_{\p M}^E( \overline{M})$ is a homeomorphism, it holds that 
$$
z \in \gamma_{q,-\eta}((0,\tau_{exit}(q,-\eta))) \setminus  J(q,\eta)
$$ 
if and only if $R_{\p M}^E(z)$ has a neighborhood $V \subset R_{\p M}^E(\overline M)$ such that $\Sigma(q,\eta) \cap V$
is homeomorphic to interval $[0,1)$ or $(0,1)$, here $\Sigma(q,\eta)$ is as in \eqref{eq:image_of_open_geo_1}. Therefore, for a given $r=R_{\p M}^E(p)$ and $(q,\eta)\in r$ the data \eqref{bigger geodesic measurment data} determines the set $R_{\p M}^E(J(q,\eta))$ and it holds that 
$$
(q,\eta) \in K(p) \hbox{ if and only if } r \notin R_{\p M}^E(J(q,\eta)).
$$

Finally we will verify the equation \eqref{eq:almost_good_dir_are_equiv} in the case of $p \in M_1$. Denote $r=R_{\p M_1}^E(p)$. By Theorem \ref{th:topology} the map $\Psi$ is a homeomorphism. With \eqref{eq:equivalence_of_biger_data} this implies that for every $(q, \xi)\in \p_-SM_1$ 
$$
\widehat{D \Phi}(\Sigma(q,\xi))=\Sigma(\phi (q),D\Phi \xi)
$$
and the self-intersection points of $\Sigma(y,\xi)$ are mapped onto self-intersection points of $\Sigma(\phi(y),D\Phi\xi)$ under the map $\widehat{D \Phi}$. Therefore, $D\Phi (K(p))=K(\Psi(p))$.

\medskip
Suppose next that $p \in \p M$.  Assume that there exists a sequence $(\xi_j)_{j=1}^\infty \subset I$ that converges to a vector $\xi \in S_p\p M$ in $S_pN$. Choose a sequence $(t_j)_{j=1}^\infty $ such that
$$
\exp_p(t_j\xi_j)=p \hbox{ and } 0<t_j.
$$
Due to Lemmas \ref{LE:2_fundamental_form} and \ref{Le:exit_time_func_is_cont} we may assume that $t_j \to 0$ as $j \to \infty$. But then we have a contradiction with the injectivity radius of $p$. Therefore, $\overline I$ is contained in the interior of $\p_+ S_pM$. Then by a similar argument as in the case of $p \in M$ we have shown that $S_pN \setminus I$ is open and dense. 

Let $(y,\xi) \in \p_- SM$.  We will use a short hand notation
\begin{eqnarray}
\label{eq:geodesic_segments1}
\overline \Sigma(y,\xi)&:=& \{R_{\p M}^E(q); q \in \overline{M}, \: (y,\xi) \in R_{\p M}^E(q))\}
\\
&=&R^E_{\p M}(\gamma_{y,-\xi}([0,\tau_{exit}(y,-\xi)]))\subset 2^{C(\partial SM)} \nonumber,
\end{eqnarray}
for the image of the geodesic segment $\gamma_{y,-\xi}([0,\tau_{exit}(y,-\xi)])$ under the map $R^E_{\p M}$. 
Denote $r = R_{\p M}^E(p)$. Then for a vector $(q,\eta)\in r$ it holds that $(q,\eta)\in K(p)$ if and only if $q \neq p$. Thus the set 
$$
K(p)=\{(\gamma_{p,\xi}(\tau_{exit}(p,\xi)),\dot \gamma_{p,\xi}(\tau_{exit}(p,\xi)))\in \p_-SM;  \: \xi \in (\p_+ S_pM)^{int} \setminus I\}.
$$  
is not empty.

Now we will verify the equation \eqref{eq:almost_good_dir_are_equiv} in the case of $p \in \p M_1$. Since the map $\Psi$ is a homeomorphism we have by the data \eqref{eq:equivalence_of_biger_data} that
\begin{equation}
\label{eq:image_of_closed_geo}
\widehat{D \Phi} (\overline \Sigma(y,\xi))=\overline \Sigma(\phi (y),D \Phi\xi) \hbox{ for all } (y,\xi)\in \p_-SM_1.
\end{equation}
Therefore, $D\Phi(K(p))=K(\Psi(p))$ and the equation \eqref{eq:almost_good_dir_are_equiv} is proven.
\end{proof}

In the next Lemma we will show that we can find the set of conjugate directions with respect to $p$ from data \eqref{bigger geodesic measurment data} and there exist lots
of $(q,\eta) \in r:=R_{\p M}^E(p)$ such that $\eta$ is not a conjugate direction with respect to $p$. If $(q,\eta) \in K(p)$ is not a conjugate direction with respect to $p$, we will later construct coordinates for $p$ such that $(n-1)$-coordinates are given by $\eta$.
 
\begin{Lemma}
\label{Le:bad_directions}
Let $(N,g)$ and $M$ be as in the Theorem \ref{th:main}. If $r \in R_{\p M}^E(\overline M)$ is given and $p\in \overline{M}$ is the unique point for which $R_{\p M}^E(p)=r$, we define the set of "good`` directions
\begin{equation}
\label{eq:good_directions}
\begin{array}{ll}
K_G(p):= &
\{(q,\eta) \in K(p); \eta \hbox{ is not a conjugate direction  with respect to $p$} \}\subset \p SM.
\end{array}
\end{equation} 
Then $K_G(p)$ is not empty, $\pi(K_G(p))\subset \p M$ is open. Moreover $r$ and the data \eqref{bigger geodesic measurment data} determine the set $K_G(p)$.

More precisely if $(N_i,g_i)$ and $M_i$ are as in Theorem \ref{th:main} such that \eqref{eq:collar_neib} and \eqref{eq:equivalence_of_biger_data} hold, then $ \hbox{ for every } p \in \overline{M_1}$ the set $K_G(p)\neq \emptyset$, $\pi(K_G(p))\subset \p M_1$ is open and 
\begin{equation}
\label{eq:the_set_of_good_dir_are_the_same}
D\Phi (K_G(p))=K_G(\Psi(p)).
\end{equation}

\end{Lemma}

\begin{proof}
Suppose first that $p \in M$. Let $\delta:SN \to [0,\infty]$ be the cut distance function, i.e.,
$$
\delta(y,\eta):=\sup\{t>0;d_g(\gamma_{y,\eta}(t),y)=t\}.
$$

Let $\xi \in S_pN$ be such a unit vector that $\gamma_{p,\xi}$ is a shortest geodesic from $p$ to the boundary $\p M$. By Lemma 2.13 of \cite{Katchalov2001} it holds that 
$$
\delta(p,\xi)>dist_g(p, \p M).
$$ 
Therefore, $(\gamma_{p,\xi}(\tau_{exit}(p,\xi)),\dot \gamma_{p,\xi}(\tau_{exit}(p,\xi)))$ is not a conjugate direction with respect to $p$. Moreover there exists an open neighborhood $V\subset S_pN$ of $\xi$ such that for any $\xi' \in V$ the vector $(\gamma_{p,\xi'}(\tau_{exit}(p,\xi')),\dot \gamma_{p,\xi'}(\tau_{exit}(p,\xi')))$ is not a conjugate direction. Let $I \subset S_pN$ be defined as in \eqref{eq:self_intersection_directions}. Denote
$$
V_c:=\bigg\{\xi \in S_pN; D\exp_p\bigg|_{\tau_{exit}(p,\xi)\xi} \hbox{ is not singular }\bigg\}
$$
that is open and non-empty, since $V \subset V_c$. Therefore, by the Lemma \ref{Le:non_intersection_directions} it holds that the set $V_c \cap (S_pN \setminus I)=V_c \setminus I$ is open and non-empty. Thus
$$
K_G(p)=\bigg\{(\gamma_{p,\xi}(\tau_{exit}(p,\xi)),\dot \gamma_{p,\xi}(\tau_{exit}(p,\xi)))\in (\p_-SM\setminus S_pN);  \xi \in V_c \setminus I \bigg\}
$$ 
is not empty and moreover $\pi(K_G(p)) \subset \p M$ is open. 

\medskip
Next we show that for given $r=R_{\p M}^E(p)$ the data \eqref{bigger geodesic measurment data} determines the set $K_G(p)$, if $r\in R^{E}_{\p M}(M)$. Fix $(q,\eta)\in K(p).$ Let $t_p \in (0,\tau_{exit}(q,-\eta)]$ be such that $\gamma_{q,-\eta}(t_p)=p$. Consider the collection $H(q,\eta)$ of $C^{\infty}$-curves $\sigma:(-\epsilon,\epsilon)\to \p_- S M, \:\sigma(s)=(q(s),\eta(s))$ that satisfy the following conditions
\begin{enumerate}
\item[(H1)] $\sigma(0)=(q,\eta)$, 
\item[(H2)] $r\in \Sigma(q(s),\eta(s))$ for all $s \in (-\epsilon,\epsilon)$
\item[(H3)] $r \notin R_{\p M}^E(J(q(s),\eta(s)))$.
\end{enumerate}
Here we use the notations from Lemma \ref{Le:non_intersection_directions}
$$
\Sigma(q(s),\eta(s))=R^E_{\p M}(\gamma_{q(s),-\eta(s)}((0,\tau_{exit}(q(s),-\eta(s))))),
$$
and 
$J(q(s),\eta(s))$ is the collection of self-intersection points of the geodesic segment
\\
$\gamma_{q(s),-\eta(s)}:(0,\tau_{exit}(q(s),-\eta(s))) \to \overline{M}$. Choose a $C^{\infty}$-curve $s \mapsto \sigma(s)$ on $\p_-SM$, such that (H1) holds. By \eqref{eq:geodesic_segments1} we can check, if the property (H2) holds. By the proof of Lemma \eqref{Le:non_intersection_directions} we can check, if property (H3) holds  for the curve $\sigma$. Therefore, for a given $r\in R_{\p M}^E(M)$ and $(q,\eta) \in K(p)$ we can check the validity of properties (H1)--(H3) for any smooth curve $\sigma$ on $\p_-SM$ with  the data \eqref{bigger geodesic measurment data}. 

\medskip
Next we show that the collection $H(q,\eta)$ is not empty. Let $\xi  \in S_pN$ be the vector that satisfies $(\gamma_{p,\xi}(t_p),\dot \gamma_{p,\xi}(t_p))=(q,\eta)$. Then $\tau_{exit}(p,\xi)=t_p$. Let $w \in T_pN$, $w \perp \xi$ and $\epsilon>0$ be small enough. We define
\begin{equation}
\label{eq:definition_G_1}
w(s):=\frac{\xi+sw}{\|\xi+sw\|_g} \hbox{ and } f(t,s):=\exp_p(t\tau_{exit}(p,w(s))w(s)), \: s\in (-\epsilon,\epsilon), t\in [0,1].
\end{equation}
Notice that the curve $s\mapsto f(1,s)\in \p M$ is smooth by Lemma \ref{Le:exit_time_func_is_cont}. We conclude that a curve $\sigma$ can be for instance defined as follows
\begin{equation}
\label{eq:definition_G_2}
\sigma(s)=(q(s),\eta(s)):=\bigg(f(1,s),\tau_{exit}(p,w(s))^{-1}\frac{\p }{\p t}f(t,s)\bigg|_{t=1}\bigg), \: s\in (-\epsilon,\epsilon),
\end{equation}
where $\epsilon>0$ is small enough. Observe that the curve $\sigma(s)$ satisfies conditions (H1) and (H2) automatically. The condition (H3) is valid by the Lemma \ref{Le:non_intersection_directions}, since $(q,\eta)\in K(p)$. We also note that
\begin{equation}
\label{eq:definition_G_3}
\frac{d}{ds}q(s)\bigg|_{s=0}=D \exp_p\bigg|_{t_p\xi}\bigg(\frac{d}{ds}\tau_{exit}(p,w(s))\bigg|_{s=0}\xi+t_p\dot{w}(0)\bigg) 
\end{equation}
and
\begin{equation}
\label{eq:definition_G_4}
\dot{w}(0)=w.
\end{equation}
\medskip

Now we show the relationship between conjugate directions and curves $\sigma \in H(q,\eta)$. Let us first consider some notations we will use. Let $c:(a,b) \to N$ be a smooth path and $V$ a smooth vector field on $c$. By this we mean that $t \mapsto V(t)\in T_{c(t)}N$. We will write $D_tV$ for the covariant derivative of $V$ along the curve $c$. The properties of operator $D_t$ are considered for instance in Chapter 4 of \cite{Lee2}. 
We recall that locally $D_t V$ is defined by the formula
\begin{equation}
\label{eq:covariant_der_of_vec_field_on _curve}
D_t V(t)=\bigg(\frac{d}{dt} V^k(t)+V^j(t)\dot{c}^i(t)\Gamma_{ji}^k(c(t))\bigg)\frac{\p}{\p x^k}(c(t)),
\end{equation}
where $\Gamma_{ji}^k$ are the Christoffel symbols of the metric tensor $g$.

\color{black}
Suppose first that $\gamma_{q,-\eta}(t_p)=p$ is a conjugate point of $q$ along $\gamma_{q,-\eta}([0,\tau_{exit}(q,-\eta)])$. Then there is a non-trivial Jacobi field $J$ on $\gamma_{p,\xi}$ that vanishes at $t=0$ and at $t=t_p$. Then 
$$
D_t J(t)\bigg|_{t=0}\neq 0 \hbox{ and }D_t J(t)\bigg|_{t=0} \perp \xi.
$$
Here $D_t$ operator is defined on $\gamma_{p,\xi}(t)$. Define a curve $\sigma(s)=(q(s),\eta(s)) \in H(q,\eta)$ by \eqref{eq:definition_G_1} and \eqref{eq:definition_G_2} with $w:=D_t J(t)\bigg|_{t=0}$. We will show that 
\begin{equation}
\label{eq:char_of_conjugate_point2}
\frac{d}{ds}q(s)\bigg|_{s=0}=0 \hbox{ and } (D_s\eta(s)^T\bigg|_{s=0})^{T}\neq 0.
\end{equation}
Here $D_s$ is defined on $q(s)$. Since the vector $w$ is the velocity of $J$ at $0$ and $J$ is a Jacobi field that vanishes at $0$ and $t_p$ we have by \eqref{eq:definition_G_3} and \eqref{eq:definition_G_4} that
$$
\frac{d}{ds}q(s)\bigg|_{s=0}=\frac{d}{ds}\tau_{exit}(p,w(s))\bigg|_{s=0}\eta.
$$
Since $s \mapsto q(s)$ is a curve on $\p M$ and $\eta\neq 0$ is not tangential to the boundary, it must hold that  
\begin{equation}
\label{eq:der_of_q}
\frac{d}{ds}\tau_{exit}(p,w(s))\bigg|_{s=0}=0. 
\end{equation}
Thus 
\begin{equation}
\label{eq:der_of_q_is_0}
\frac{d}{ds}q(s)\bigg|_{s=0}=0.
\end{equation} 
Recall that the Jacobi field $J$ satisfies due to \eqref{eq:der_of_q}
$$
J(t)=t D\exp_{p}\bigg|_{t\xi} w=\frac{d}{ds} f(t',s)\bigg|_{s=0}, \: t\in [0,t_p], \: t'=\frac{t}{t_p}.
$$ 
We will use the notation $D_{t'}$ for the covariant derivative on the curve $\gamma_{p,t_p\xi}(t')$. Then by \eqref{eq:definition_G_2}--\eqref{eq:definition_G_4}, \eqref{eq:der_of_q}--\eqref{eq:der_of_q_is_0} and Symmetry Lemma (\cite{Lee2}, Lemma 6.3) we have
$$
D_s\eta(s)\bigg|_{s=0}=t_p^{-1}D_s\frac{\p }{\p t'}f(t',s)\bigg|_{t'=1,s=0}=t_p^{-1}D_{t'}\frac{\p }{\p s}f(t',s)\bigg|_{t'=1,s=0}
$$
$$
= t_p^{-1}D_t J(t)\bigg|_{t=t_p}\neq 0,
$$
since $J$ is not a zero field. Therefore, we conclude that $D_s\eta(s)\bigg|_{s=0}\neq 0$.  Moreover since $\|\eta(s)\|_g\equiv 1$ we have $\eta(s)\perp D_s\eta(s)$. Suppose first that $(D_s\eta(s)\bigg|_{s=0})^T= 0$, then $D_s\eta(s)\bigg|_{s=0}$ is normal to the boundary and thus $\eta(0)$ is tangential to the boundary. This is a contradiction since $\eta(0)=\eta$, that is not tangential to the boundary. Thus we conclude that $(D_s\eta(s)\bigg|_{s=0})^T\neq 0$. Write 
$$
\eta(s)=c(s)\nu(q(s))+\eta(s)^T, \hbox{ where $c$ is some smooth function}.
$$
Let $x\mapsto (z^j(x))_{j=1}^n$ be the boundary normal coordinates  near $q$. Here we consider that $z^n(x)$ represents the distance of $x$ to $\p M$. Then the coefficient functions of $\nu$ are $V^k=\delta^n_k$. Therefore for every $k\in \{1,\ldots,n\}$ the observation $\dot{q}(0)=0$ yields
$$
\bigg(D_s \nu(s)\bigg|_{s=0}\bigg)^k=\bigg(\frac{d}{ds} V^k(s)+V^j(s)\dot{q}^i(s)\Gamma_{ji}^k(q(s))\bigg)\bigg|_{s=0}=0,
$$
here $\Gamma_{ji}^k$ are the Christoffel symbols of metric $g$ in boundary normal coordinates. Thus
\begin{equation}
\label{eq:tangential_comp_of_eta}
(D_s \eta(s)\bigg|_{s=0})^{T}=
c(0)D_s \nu(q(s))\bigg|_{s=0}+(D_s \eta(s)^T\bigg|_{s=0})^{T}=(D_s \eta(s)^T\bigg|_{s=0})^{T}.
\end{equation}
Therefore, \eqref{eq:char_of_conjugate_point2} is valid.

\medskip
Next we consider curves $\sigma(s)=(q(s),\eta(s))\in H(q,\eta)$. Notice that for every $\sigma(s)$ there exists $\epsilon>0$ and a smooth function $s\mapsto a(s)\in (0,\infty), \: s\in (-\epsilon, \epsilon)$ such that $a(s)\to t_p$ as $s\to 0$ and 
\begin{equation}
\label{eq:variation Gamma}
\Gamma(s,t):=\exp_{q(s)}(-ta(s)\eta(s)), \: t\in [0,1], \: s\in (-\epsilon, \epsilon)
\end{equation}
is a geodesic variation of $\gamma_{q,-\eta}([0,t_p])$ that satisfies $\Gamma(s,1)=p$. Let us consider $s \mapsto(q(s),-ta(s)\eta(s)), t\in [0,1]$ as a smooth curve on $TN$ and $\exp:TN \to N$. We use a short hand notation $-tW(s):=-ta(s)\eta(s)$. We use coordinates $(z^j,v^j)_{j=1}^n$ for $\pi^{-1}U \subset TM$, where $U$ is the domain of coordinates $(z^i)_{i=1}^n$. Let $(y^i)_{i=1}^n$ be coordinates at $\exp((q(0),-tV(0)$. 
Then the variation field 
\begin{equation}
\label{eq:def_of_variation_field_V}
\begin{array}{ccc}
V(t)&:=&\frac{\p}{\p s}\Gamma(s,t)\bigg|_{s=0}=\bigg(\frac{\p \exp^j }{\p z^i}\dot q^i(0)-t\frac{\p \exp^j }{\p v_i}\dot W^i(0)\bigg)\frac{\p }{\p y^j}
\\
&=&\bigg(\frac{\p \exp^j }{\p z^i}\dot q^i(0)-tD\exp_p\bigg|_{-tW(0)}\dot W(0)\bigg)\frac{\p }{\p y^j},
\end{array}
\end{equation}
is a Jacobi field that vanishes at $t=1$. By \eqref{eq:variation Gamma} we have
$$
V^j(0)=\dot q^j(0).
$$ 
Therefore by \eqref{eq:def_of_variation_field_V} the point $p$ is a conjugate point of $q$ along $\gamma_{q,-\eta}([0,\tau_{exit}(q,-\eta)])$, if there exists a curve $\sigma\in H(q,\eta)$ that satisfies 
$$
\dot q(0)=0 \hbox{ and } \dot W(0)\neq 0.
$$
By the definition \eqref{eq:covariant_der_of_vec_field_on _curve}  of the operator $D_s$  and the vector field $W(s)$ this is equivalent to
\begin{equation}
\label{eq:char_of_conjugate_point}
\dot q(0)=0 \hbox{ and } D_s(a(s)\eta(s))\bigg|_{s=0}\neq 0.
\end{equation}

Suppose that there exists such a curve $\sigma(s)=(q(s),\eta(s))\in H(q,\eta)$ for which \eqref{eq:char_of_conjugate_point} is valid. Since $V$ is a Jacobi field that vanishes at $1$, we have by  \eqref{eq:def_of_variation_field_V},  \eqref{eq:char_of_conjugate_point} and the definition of $\sigma(s)$ that
$$
0=V(1)=-D\exp_{q}\bigg|_{-t_p\eta}\bigg(D_s(a(s)\eta(s))\bigg|_{s=0}\bigg)
$$
$$
=-\dot{a}(0)\xi- D\exp_{q}\bigg|_{-t_p\eta}\bigg(t_pD_s\eta(s)\bigg|_{s=0}\bigg).
$$
Thus $\dot{a}(0)=0$ and by \eqref{eq:char_of_conjugate_point} it holds that $D_s\eta(s)\bigg|_{s=0}\neq 0$, which implies that $(D_s\eta(s)\bigg|_{s=0})^T\neq 0$. Therefore, by similar computations as in \eqref{eq:tangential_comp_of_eta} we have 
\\
$(D_s\eta(s)^T\bigg|_{s=0})^T\neq 0$. We conclude that, if there exists a curve $\sigma \in H(q,\eta)$ such that \eqref{eq:char_of_conjugate_point} holds then, $p$ is a conjugate point to $q$ on $\gamma_{q,-\eta}$ and moreover, 
$$
(D_s \eta(s)^T\bigg|_{s=0})^{T}\neq 0.
$$

To summarize, for a given $r \in R_{\p M}^E(M)$ a vector $(q,\eta) \in K(p)$ is a conjugate direction of $p$ along geodesic $\gamma_{q,-\eta}$ if and only if there exists $\sigma \in H(q,\eta)$ such that \eqref{eq:char_of_conjugate_point2} is valid. Moreover since $s \mapsto q(s)$ is a curve on the boundary $\p M$ we can check the validity of \eqref{eq:char_of_conjugate_point2} for any $\sigma \in H(q,\eta)$ by data \eqref{bigger geodesic measurment data}. Therefore for given $r\in R_{\p M}^E(M)$ the data \eqref{bigger geodesic measurment data} determines the set $K_G(p)$. 

\medskip 
Next we will prove equation \eqref{eq:the_set_of_good_dir_are_the_same} in the case of $p \in M_1$. Let $(q,\eta)\in K(p)$. Suppose that $D \Phi \eta\in K(\Psi(p))$ is not in $K_G(\Psi(p))$. Then there exists a smooth curve $s \mapsto\sigma(s)=(\widetilde q(s),\widetilde \eta(s)) \in \p_- SM_2$ that satisfies the properties (H1)--(H3) and for which \eqref{eq:char_of_conjugate_point2} is valid. 

As it holds that $\Psi$ is a homeomorphism and $\Phi$ is a diffeomorphism that preserves the boundary metric (in the sense of \eqref{eq:Phi_preserves_boundary_inner_product}) the curve 
$$
s \mapsto (\phi^{-1}(\widetilde q(s)), D \Phi^{-1} \widetilde \eta(s)):=(q(s),\eta(s))
$$ 
satisfies also the conditions (H1)--(H3) and moreover 
$$
\dot{q}(0)=0 \hbox{ and }(D_s \eta(s)^T\bigg|_{s=0})^T=D \Phi^{-1} (D_s \widetilde\eta(s)^T\bigg|_{s=0})^T \neq 0.
$$
Thus \eqref{eq:char_of_conjugate_point2} is valid for the curve $(q(s),\eta(s))$, which implies that $(q,\eta) \notin K_G(p)$. Thus \eqref{eq:the_set_of_good_dir_are_the_same} is valid in the case of $p \in M_1$.

\medskip
Suppose then that $p \in \p M$. It follows from the Lemma \ref{Le:minimizing_geo_from_boundary} that the set $K_G(p)$ is not empty and
$$
K_G(p)=\bigg\{(\gamma_{p,\xi}(\tau_{exit}(p,\xi)),\dot \gamma_{p,\xi}(\tau_{exit}(p,\xi)))\in \p_-SM ; \; \xi \in \big(V_c \setminus (I \cup S\p M) \big)\bigg\}.
$$ 
Let $(q,\eta) \in K_G(p)$. Let $\xi \in S_pN$ be the unique vector that satisfies $\eta =\dot \gamma_{p,\xi}(\tau_{exit}(p,\xi)).$ Since $D\exp_p$ is not singular at $\tau_{exit}(p,\xi)\xi$, we can choose an open neigborhood $U$ of $\tau_{exit}(p,\xi)\xi$ such that for every $v \in U$ the differential map $D \exp_p\bigg|_{v}$ is not singular and $\exp_p(U)$ is an open neighborhood of $q$. Since $q \neq p$ it holds that $\tau_{exit}(p,\xi) >0$. Therefore we may assume that $U\cap T\p M = \emptyset$ and $U \subset V_c \setminus I$, since $I$ is closed. 
Thus,
$$
\exp_p(U)\cap \p M=\{\gamma_{p,v}(\tau_{exit}(p,v)); \: v\in h(U)\}\subset \pi(K_G(p)),
$$ 
where $h(v):=\frac{v}{\|v\|_g}$ and also in this case we deduce that $\pi(K_G(p)) \subset \p M$ is open. 

Next we show that $K_G(p)$ is determined by $r$ and data \eqref{bigger geodesic measurment data}, if $r \in R_{\p M}^E(\p M)$. Fix $(q,\eta)\in K(p).$ We note that $q\neq p$ and by Lemma \eqref{LE:2_fundamental_form} the vector $\eta$ is not tangential to the boundary. Let $t_p \in (0,\tau_{exit}(q,-\eta)]$ be such that $\gamma_{q,-\eta}(t_p)=p$. Consider the collection $H'(q,\eta)$ of $C^{\infty}$-curves $\sigma:(-\epsilon,\epsilon)\to \p_- S M, \:\sigma(s)=(q(s),\eta(s))$ that satisfy the following conditions
\begin{enumerate}
\item[(H1')] $\sigma(0)=(q,\eta)$, $q(s) \neq p, \hbox{ for any } s \in (-\epsilon,\epsilon)$
\item[(H2')] $r\in \overline \Sigma(q(s),\eta(s))$ for all $s \in (-\epsilon,\epsilon)$.
\end{enumerate}
Then by similar proof as in the case of $p \in M$ we can show that $H'(q,\eta) \neq \emptyset$, $H'(q,\eta)$ is determined by $r$ and data \eqref{bigger geodesic measurment data} and $(q,\eta)$ is a conjugate direction with respect to $p$ if and only if there exists a curve $\sigma \in H'(q,\eta)$ for which equation \eqref{eq:char_of_conjugate_point2} is valid. 

The equation \eqref{eq:the_set_of_good_dir_are_the_same} is also valid in the case of $p \in \p M_1$ by a similar argument as in the case of $p \in M_1$.
\end{proof}

Now we are ready to consider the map defined in \eqref{eq:n-1,coordinate_map}. Choose $(q,\eta) \in K_G(p)$ and let $t_p<0$ be such that $\exp_q(t_p \eta)=p$. Let $V_q\subset M$ be a neighborhood of $p$. 
Assuming that this neighborhood is small enough, there is a neighborhood $U_q\subset T_qN$ of $t_p \eta$ such that the exponential map 
$$
\exp_q:U_q\to V_q
$$
is a diffeomorphism.

We emphasize that the mapping 
$$
\Theta_q(z)=\frac 1{\|\exp_q^{-1}(z)\|_g} \exp_q^{-1}(z)\in S_qN ,\quad z\in V_q .
$$
depends on the neighborhood $U_q$ of $t_p \eta$.

\begin{Lemma}
\label{Le:ker_of_Dtheta}
Let $p,q \in N, \: p\neq q$. Let $\eta\in S_qN$ and $t>0$ be such that $\exp_q(t\eta)=p$. Suppose that $D\exp_q$ is not singular at $t\eta$ and denote $D\exp_q|_{t\eta}\eta=:\xi \in S_pN$. Then ker$(D\Theta_q(p)) =\hbox{span}(\xi)$.
\end{Lemma}
\begin{proof}
Let $v \in T_pN$. Then 
$$
D\Theta_q(p)v=\frac{D\exp_q^{-1}v}{t}-\eta\frac{\langle D\exp_q^{-1}v, \eta \rangle_g}{t},
$$
and the claim follows.
\end{proof}
In the next Lemma we will show that for given $r=R^{E}_{\p M}(p)$ and $(q,\eta) \in K_G(p)$ the data \eqref{bigger geodesic measurment data} determine the map $\Theta_q$.
\begin{Lemma}
\label{Le:n-1_coordinates}
Let $(N,g)$ and $M$ be as in the Theorem \ref{th:main}. If $r \in R_{\p M}^E(\overline M)$ is given and $p\in \overline M$ is the unique point for which $R_{\p M}^E(p)=r$, then for any $(q,\eta) \in K_G(p)$ there exists a neighborhood $V_q$ of $p$ and a neighborhood of $U_q$ of $t_p\eta$, where $\exp_q(t_p\eta)=p$, such that $\exp_q^{-1}:V_q \to U_q$ is well defined. Moreover, the map $\Theta_q:V_q \to S_qN$ is smooth and well defined.

The set $R^E_{\p M}(V_q)$ and the map $\Theta_q \circ (R^E_{\p M})^{-1}:R^E_{\p M}(V_q) \to S_qN $ are determined from the data \eqref{bigger geodesic measurment data} for given $r \in R_{\p M}^E(\overline M)$ and $(q,\eta)\in K_G(p)$.

More precisely, if $(N_i,g_i)$ and $M_i$ are as in Theorem \ref{th:main} such that \eqref{eq:collar_neib} and \eqref{eq:equivalence_of_biger_data} hold. Then for a given $r=R^E_{\p M}(p) \in R_{\p M}^E(\overline{M_1})$ and $(q,\eta)\in K_G(p)$ it holds that
\begin{equation}
\label{eq:coordinates_match}
\Theta_{\phi(q)}(\Psi(z))= D\Phi (\Theta_{q}(z)), \: z \in V_q.
\end{equation}
\end{Lemma}
\begin{proof}
Assume first that $p \in M$. Let $(q,\eta)\in K_G(p)$, then the existence of sets $V_q$ and $U_q$ follows. The map $\Theta_q$ is well defined and smooth since $\Theta_q=h\circ\exp_q^{-1}$. Here $h:T_qN\setminus \{0\} \to S_qN, \: h(v):= \frac{v}{\|v\|_g}$ is smooth and well defined since $q\neq p$.

\medskip
Next we will show that the set $V_q$ and the map $\Theta_q$ are determined from the data \eqref{bigger geodesic measurment data}, if $(q,\eta)\in K_G(p)$ is given. Let $V \subset M$ be a neighborhood of $p$ and $\widetilde U \subset T_qN$ a neighborhood of $t_p\eta$. Denote $U:=h(\widetilde U) \subset S_qN$. Since $\eta \in U$, we may assume that for all $z\in V$ the set $ R^E_{\p M}(z) \cap U\subset S_qN$ is not empty.  We define a set valued mapping $P_q:V \to 2^{S_qN}$ by formula
$$
P_q(z)= R^E_{\p M}(z) \cap U.
$$
Then $P_q$ is well defined and for given $r$ and $(q,\eta)\in K_G(p)$ we can recover $P_q$ from \eqref{bigger geodesic measurment data}. We claim that, if $V$ and $U$ are small enough, then
\begin{equation}
\label{eq:neihborhood_og_geo}
\hbox{ for all $z \in V$ the set } R^E_{\p M}(z) \cap U\subset S_qN \hbox{ has a cardinality of } 1.
\end{equation}
Therefore, $P_q,$ coincides with $\Theta_q$ in $V.$ We will show that if \eqref{eq:neihborhood_og_geo} is not valid, then we end up in a  contradiction with the assumption $(q,\eta)\in K_G(p)$. We will divide the proof into two parts.

Assume first that there exists a sequence $(\eta_j)_{j=1}^\infty \subset R^E_{\p M}(p) \cap S_qN$ that converges to $\eta$ and $\eta_j \neq \eta$ for any $j \in \N$. Choose $t_j<0$ such that $\exp_q(t_j\eta_j)=p$. Due to \eqref{eq:exit_time_func_is_bdd} we may assume that $t_j$ converges to some $t$. Therefore, by the continuity of the exponential map, we have 
$$
p=\lim_{j\to \infty }\exp_q(t_j\eta_j)=\exp_q(t \eta).
$$
Since $(q,\eta) \in K(p)$ it must hold that $t=t_p$. As $p=\exp_q(t_j\eta_j)$ for every $j \in \N$ and $t_j\eta_j \to t\eta$ in $T_qN$, we end up with a contradiction to the assumption $(q,\eta) \in K_G(p)$.

Suppose next that there exist sequences $(z_j)_{j=1}^\infty\subset \overline{M}$ and $\xi^1_j,\xi_j^2 \in R^E_{\p M}(z_j) \cap S_qN, \: \xi_j ^1 \neq \xi_j^2$ for every $j \in \N$ such that for $i\in \{1,2\}$ holds
$$
z_j \longrightarrow p, \hbox{ and }  \xi_j ^i \longrightarrow \eta \hbox{ as } j \longrightarrow \infty.
$$
Choose $t_j^1\in (-\tau_{exit}(q,-\xi_j^1),0), t_j^2\in (-\tau_{exit}(q,-\xi_j^2),0), \: j \in \N$ such that for $i\in \{1,2\}$, holds $\exp_q(t_j^i\xi_j^i)=z_j$. Again by \eqref{eq:exit_time_func_is_bdd} we may assume that $t_j^i$ converges to $t^i \in (-\tau_{exit}(q,-\xi),0)$ as $j \to \infty$ for $i \in \{1,2\}$, respectively. Therefore, we have 
$$
p=\lim_{j \to \infty}z_j=\lim_{j\to \infty }\exp_q{(t_j^i\xi_j^i)}=\exp_{q}(t^i\eta).
$$
Since $(q,\eta) \in K(p)$ it must hold that $t^i=t_p$ for both $i\in \{1,2\}$. Therefore, $t_j^i\xi_j^i \to t_p\eta$ in $T_qN$ as $j \to \infty$ for both $i\in \{1,2\}$. Since we assumed that $\xi_j^1\neq \xi_j^2$ for any $j \in \N$, we have shown that mapping $\exp_q$ is not one--to--one near $t_p\eta$. Hence we end up again in a contradiction with the assumption $(q,\eta)\in K_G(p).$ 

Thus we have proven the existence of such sets $V$ and $U$ for which \eqref{eq:neihborhood_og_geo} is valid. Using data \eqref{bigger geodesic measurment data} we can verify for a given $r$ and the neighborhoods $U$ of $\eta$ and $R^{E}_{\p M}(V)$ of $r$ whether the property \eqref{eq:neihborhood_og_geo} is valid. We will denote by $U_1 \subset S_qN$ a neighborhood of $\eta $ and $V_q$ a neighborhood of $p$ that satisfy \eqref{eq:neihborhood_og_geo}. 

\medskip
The next step is to prove that for any $r=R_{\p M_1}^E(p)\in R_{\p M_1}^E(M_1)$ and $(q,\eta)\in K_G(p)$ the equation \eqref{eq:coordinates_match} is valid. Choose neighborhoods $U_1 \subset S_qN$ of $\eta$ and $V_q \subset M_1$ of $p$ for which \eqref{eq:neihborhood_og_geo} is valid. By \eqref{eq:the_set_of_good_dir_are_the_same} it holds that $(\phi(q),D\Phi \eta)\in K_G(\Psi(p))$. Since $\Psi$ is a homeomorphism $\Psi(V_q)$ is a neighborhood of $\Psi(p)$ and due to \eqref{eq:Phi_preserves_boundary_inner_product} $D\Phi(U_1)\subset S_{\phi(q)}N_2$ is a neighborhood of $D\Phi \eta$. Therefore \eqref{eq:neihborhood_og_geo} is valid for $\Psi(V_q)$ and $D\Phi(U_1)$. Let $z \in V_q$ and  $\{\xi\}=  R^E_{\p M}(z) \cap U_1$. Then $\{D\Phi \xi\} =R^E_{\p M}(\Psi (z)) \cap D \Phi(U_1)$ and the equation \eqref{eq:coordinates_match} follows.

\medskip 
Finally we assume that $p \in \p M$. We notice that by Lemma \eqref{LE:2_fundamental_form} it holds $K_G(p) \cap T\p M =\emptyset$ and by the definition of $K_G(p)$ we have $p \notin \pi(K_G(p))$. Therefore, we can choose  for every $(q,\eta) \in K_G(p)$ such a neighborhood $V_q\subset \overline{M}$ of $p$ that $q \notin V_q$. The rest of the proof is similar to the case $p \in M$.
\end{proof}

\subsection{Interior coordinates}
For this subsection we assume that $p \in M$. Let $(\widetilde q, \widetilde \eta) \in K_G(p)$ be such that a similar map $\Theta_{\widetilde q}:V_{\tilde q}\to S_{\widetilde q}N$ exists for some neighborhood $V_{\widetilde q}$ of $p$. Let $v\in T_{\widetilde q}N$ and define a smooth map
\ba
\Theta_{q,\widetilde q, v}(z)=(\Theta_q(z),\bra v,\Theta_{\tilde q}(z)\rangle_g )\in S_{q}N\times \R ,\quad z\in V_q\cap V_{\tilde q}.
\ea
See Figure \eqref{fig:int_coord}. 
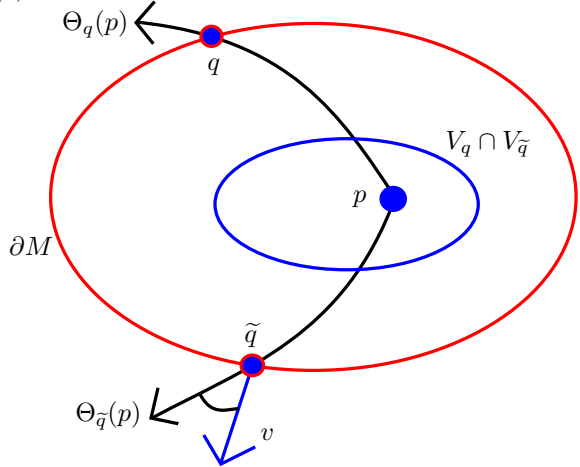
\begin{figure}[h]
 \begin{picture}(100,150)
  \put(-60,-10){\includegraphics[width=200pt]{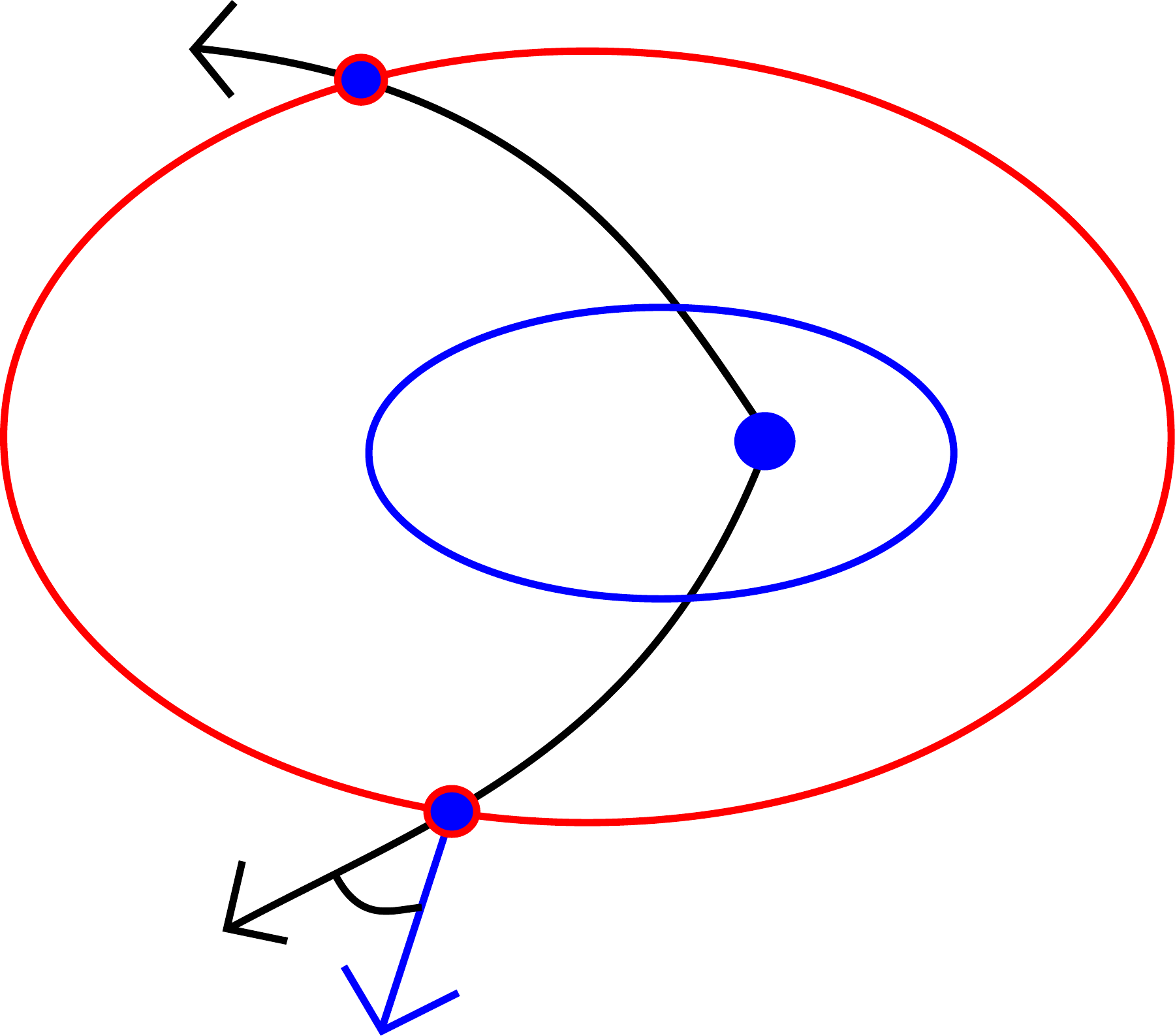}}
  \put(-55,155){$\Theta_q(p)$}
  \put(0,140){$q$}
  \put(14,37){$\widetilde q$}
  \put(20,0){$v$}
  \put(-50,7){$\Theta_{\widetilde q}(p)$}
  \put(-75,70){$\p M$}
  \put(55,90){$p$}
  \put(90,110){$V_q\cap V_{\widetilde q}$}
   \end{picture}
\caption{Here is a schematic picture about the map $\Theta_{q,\widetilde q, v}$ evaluated at point $p\in M$, where the point $p \in M$ is the blue dot and $V_q\cap V_{\widetilde q}$ is the blue ellipse. The higher red\&blue dot is $q$ and the lower is $\widetilde q$. The blue vector is the given direction $v \in T_{\widetilde q}N$.}
\label{fig:int_coord}
\end{figure}

In the next Lemma we will show that this map is a smooth coordinate map near $p$. By Inverse function theorem it holds that a smooth map $f:N \to \R^n$ is a coordinate map near $p \in N$ if and only if $\det Df(p)\neq 0$. Here $Df(p)$ is the Jacobian matrix of $f$ at $p$.

\begin{Lemma}
\label{Le:smooth_atlas}
Let $(N,g)$ and $M$ be as in the Theorem \ref{th:main}. If $r \in R_{\p M}^E(M)$ is given and $p\in M$ is the unique point for which $R_{\p M}^E(p)=r$, then there exist $(q,\eta),(\widetilde q, \widetilde \eta) \in K_G(p)$ and $v\in T_{\widetilde q}N$ such that $\Theta_{q,\widetilde q,v}$ is a smooth coordinate mapping from a neighborhood $V_{q,\tilde q,v}$ of  $p$. 


Moreover, for given $r = R^E_{\p M}(p) \in R_{\p M}^E(M)$ and $(q,\eta),(\widetilde q, \widetilde \eta) \in K_G(p)$ the data \eqref{bigger geodesic measurment data} determines $v\in T_{\tilde q}N$, a neighborhood $R_{\p M}^E(V_{q,\tilde q,v})$  of $r$ and the map $\Theta_{q,\widetilde q,v}\circ (R_{\p M}^E)^{-1}: R_{\p M}^E(V_{q,\tilde q,v}) \to (S_qN \times (\R \setminus \{0\}))$ such that $\Theta_{q,\widetilde q,v}$ a smooth coordinate map from a neighborhood $V_{q,\tilde q,v}$ of $p$ onto some open set $U \subset (S_qN \times (\R \setminus \{0\}))$.

\color{black}

More precisely, let $(N_i,g_i)$ and $M_i$ be as in Theorem \ref{th:main} such that \eqref{eq:collar_neib} and \eqref{eq:equivalence_of_biger_data} hold. Then for any $p \in M_1$ and  $(q,\eta),(\widetilde q, \widetilde \eta) \in K_G(p)$ and $v\in T_{\widetilde q}N$ the following holds: Let $\xi \in S_qN$ be outward pointing and $s\in \R$. Then
\begin{equation}
\label{eq:interior_coordinates_matc}
\Theta_{\phi(q),\phi(\widetilde q),D\Phi v}(\Psi(z))=(D\Phi \xi,s) \hbox{ if and only if } \Theta_{q, \widetilde q,v}(z)=(\xi,s)
\end{equation}
and
\begin{equation}
\label{eq:determinants_interior_match}
\det \;D \Theta_{q, \widetilde q,v}(p)\neq 0 \hbox{ if and only if } \det \; D\Theta_{\phi(q),\phi(\widetilde q),D\Phi v}(\Psi(p)) \neq 0.
\end{equation}
\end{Lemma}
\begin{proof}
Choose  $(q,\eta)\in K_G(p)$. By Lemma \ref{Le:bad_directions} the set $\pi(K_G(p))$ is open and, there exists $(\widetilde q, \widetilde \eta) \in K_G(p)$, such that 
\begin{equation}
\label{eq:cond_for_q_tilde}
\widetilde q \notin \gamma_{q,-\eta}([0,\tau_{exit}(q,-\eta)]).
\end{equation}
By Lemma \ref{Le:n-1_coordinates} there exists a neighborhood $V_{q}\cap V_{\widetilde q}$ of $p$ such that the mapping $\Theta_{q,\widetilde q,v}$ is well defined and smooth for any $v \in T_{\widetilde q} N$.   

Next we show that there exists such a vector $v\in T_{\widetilde q} N$ that the Jacobian $D\Theta_{q,\widetilde q,v}$ is invertible at $p$. By the choice of $(q,\eta)$ and $(\widetilde q, \widetilde \eta)$ there exist $t,\widetilde  t >0$ and $\xi,\widetilde \xi\in S_pN$ such that $\xi$ and $\widetilde \xi$ are not parallel and  $(\gamma_{p,\xi}(t),\dot\gamma_{p,\xi}(t))=(q,\eta)$ and $(\gamma_{p,\widetilde \xi}(\widetilde t),\dot\gamma_{p,\widetilde \xi}(\widetilde t))=(\widetilde  q,\widetilde\eta)$.
Choose linearly independent vectors $w_1, \ldots, w_{n-1}\in S_qN$ that are all perpendicular to $\eta$. Then vectors 
$$
\xi_i:=D\exp_q|_{-t \eta}w_i \in T_pN, \: i \in \{1, \ldots, n-1\}
$$
are linearly independent and perpendicular to $\xi$. Therefore, the vectors $\xi_1, \ldots, \xi_{n-1},\xi$ span $T_pN$. 
Let $v \in T_{\widetilde q}N$ be such that $\langle v , \Theta_{\widetilde q}(p)\rangle_g \neq 0$ and $\langle v , D\Theta_{\widetilde q}(p)\xi\rangle_g\neq 0$. Notice that by Lemma \ref{Le:ker_of_Dtheta} such a $v$ exists. Choose coordinates $(x^1,\ldots, x^n)$ at $p$ such that
$$
\frac{\partial}{\p x^i}(p)=\xi_i; \: i \in \{1,\ldots, n-1\} \hbox{ and } \frac{\partial}{\p x^n}(p)=\xi.
$$
In these coordinates the Jacobian matrix of $\Theta_{q,\widetilde q,v}$ at $p$ is
\begin{equation}
\label{eq:jacobian}
D\Theta_{q,\widetilde q,v}(p)=\left( \begin{array}{cc}
V & \overline{0}
\\
\overline{a} & c
\end{array}\right), 
\end{equation}
where $V$ is the Jacobian  matrix of $\Theta_q$ at $p$ with respect to $(n-1)$ first coordinates, $\overline a, \overline 0 \in \R^{n-1}$ and $c = \frac{\p}{\p x_{n}} \langle v , \Theta_{\widetilde q}(x)\rangle_g|_{x=p}$. Then by the choice of basis $\xi_1,\ldots, \xi_{n-1}, \xi$ and Lemma \ref{Le:ker_of_Dtheta} it holds that $V$ is invertible. 
Moreover
\begin{equation}
\label{eq:necessary_cond_for_v}
c = \frac{\p}{\p x^{n}} \langle v , \Theta_{\widetilde q}(x)\rangle_g\bigg|_{x=p}=\frac{d}{dt} \langle v , \Theta_{\widetilde q}(\gamma_{p,\xi}(t))\rangle_g\bigg|_{t=0}=\langle v , D\Theta_{\widetilde q}(p)\xi\rangle_g\neq 0,
\end{equation}
and we conclude that
$$
\det D\Theta_{q,\widetilde q,v}(p)=(\pm 1) c\;\det V\neq 0.
$$
This shows that there exist $(q,\eta),(\widetilde q, \widetilde \eta) \in K_G(p)$ and $v\in T_{\widetilde q} N$ such that the corresponding map $\Theta_{q,\widetilde q,v}$ is a coordinate map in some neighborhood of $p$. Since the invertibility of the Jacobian $D\Theta_{q,\widetilde q,v}(p)$ is invariant to the choice coordinates, we conclude that the map $\Theta_{q,\widetilde q,v}$ is a smooth coordinate map at $p$.

\medskip 
Now we will check that for a given $r=R^E_{\p M}(p)$ and $(q,\eta),(\widetilde q, \widetilde \eta) \in K_G(p)$ the data \eqref{bigger geodesic measurment data} determines $v\in T_{\tilde q}N$, a neighborhood $R_{\p M}^E(V_{q,\tilde q,v})$ of $r$ and the map $\Theta_{q,\widetilde q,v} \circ (R_{\p M}^E)^{-1}:R_{\p M}^E(V_{q,\tilde q,v})\to (S_qN \times \R)$. By Lemma \ref{Le:bad_directions}, the data \eqref{bigger geodesic measurment data} determines the set $K_G(p)$ and it is not empty. Choose  any $(q,\eta)\in K_G(p)$. By equation \eqref{eq:image_of_open_geo_1} we can choose $(\widetilde q,\widetilde \eta) \in K_G(p)$ such that \eqref{eq:cond_for_q_tilde} is valid. By Lemmas \ref{Le:extension_from_tangential_data_to_full_data} and \ref{Le:n-1_coordinates} we can construct the neighborhoods $V_q$, $V_{\widetilde q}$ of $p$ and the map $\Theta_{q,\widetilde q,v}:V_{q,\widetilde q}\to S_qN \times \R$, where $V_{q,\widetilde q}:=V_{ q}\cap V_{\widetilde q}$, for any $v \in T_{\widetilde q}N$. Suppose now on that the points $(q,\eta)$ and $(\widetilde q, \widetilde \eta)$ are given. By \eqref{eq:jacobian} and \eqref{eq:necessary_cond_for_v} we notice that a sufficient and necessary condition for the map $\Theta_{q,\widetilde q,v}$ to be a smooth coordinate map at $p$ is that vector $v$  is contained in $A_{\widetilde q}:=T_{\widetilde q}N \setminus ((D\Theta_{\widetilde q}(p)\xi)^\perp\cup \Theta_{\widetilde q}(p)^\perp)$ that is open and dense.

Lastly we will introduce a method to test, if $v \in T_{\widetilde q}N$ is admissible, that is, $\det D\Theta_{q,\widetilde q,v}(p)\neq 0$. By the Propostion \ref{pr:R_F is homeo} the map $R_{\p M} ^E:\overline{M} \to R_{\p M} ^E(\overline{M})$ is a homeomorphism. Thus we can determine the set 
$$
\begin{array}{lll}
B_{\widetilde q}&=&\{w \in T_{\widetilde q}N; \hbox{the map } \Theta_{q,\widetilde q,w}\circ (R_{\p M} ^E)^{-1} \hbox{ is a homeomorphism}
\\
&& \hbox{ from a neighborhood of }R_{\p M} ^E(p) \hbox{ onto some open set }V \subset \R^n\}.
\end{array}
$$
For all $w,w' \in B_{\widetilde q}$ we define a homeomorphism
$$
H_{w,w'}:U_{w,w'} \to V_{w,w'}, \quad H_{w,w'}:=\Theta_{q,\widetilde q,w}\circ \Theta_{q,\widetilde q,w'}^{-1},
$$
where $U_{w,w'}, V_{w,w'} \subset \R^n$ are open. Let 
$$
C^2_{\widetilde q}:=\{(w,w')\in B_{\widetilde q} \times B_{\widetilde q}; \: H_{w,w'} \hbox{ is smooth and } \det (D H_{w,w'}(\Theta_{q,\widetilde q,w}(p))\neq 0)\}.
$$
Thus $A_{\widetilde q} \times A_{\widetilde q} \subset  C^2_{\widetilde q}$ and the data \eqref{bigger geodesic measurment data}, with $R^{E}_{\p M}(p),(q,\eta)$ and $(\widetilde q, \widetilde \eta)$ determine $C^2_{\widetilde q}$. Let $(w,w')\in C^2_{\widetilde q}$. Then by the Inverse function theorem also $(w',w)$ is included in $C^2_{\widetilde q}$. Thus  
\begin{equation}
\label{eq:side_sets_are_empty}
((T_{\widetilde q}N \setminus A_{\widetilde q} ) \times A_{\widetilde q}  )\cap C^2_{\widetilde q}=\emptyset.
\end{equation}
Write $C_{\widetilde q}$ for the projection of $C^2_{\widetilde q}$ into $T_{\widetilde q}N$ with respect to first component. 

Choose $v \in C_{\widetilde q}$. We claim that $v$ is admissible if and only if there exists an open and dense set $A \subset T_{\widetilde q}N$ such that
\begin{equation}
\label{eq:admissibility_test}
\{v\}\times A \subset C^2_{\widetilde q}.
\end{equation}
If $v$ is admissible, then for every $w \in A_{\widetilde q}$ holds $(v,w)\in C^2_{\widetilde q}$. Therefore, $\{v\}\times A_{\widetilde q} \subset  C^2_{\widetilde q}$ and \eqref{eq:admissibility_test} follows. Suppose then that \eqref{eq:admissibility_test} is valid for some open and dense set $A$. Then $A\cap A_{\widetilde q}$ is also open, dense and $\{v\}\times(A\cap A_{\widetilde q})\subset C^2_{\widetilde q}$. Suppose that $v \in T_{\widetilde q}N \setminus A_{\widetilde q}$. Then it follows that
$$
(\{v\}\times(A\cap A_{\widetilde q}))\cap C^2_{\widetilde q} \neq \emptyset.
$$ 
But this contradicts \eqref{eq:side_sets_are_empty} and it must hold that $v \in A_{\widetilde q}$, which means that $v$ is admissible. 

If $v, w \in T_{\widetilde q}N, \: v \neq w$ are chosen and they both are admissible, then the map $H_{v,w}: U_{v,w} \to V_{v,w}$ is determined by $r$, $(q,\eta), (\widetilde q, \widetilde \eta) \in K_G(p)$ and data \eqref{bigger geodesic measurment data}. Let $U'_{v,w} \subset U_{v,w}$ be an open neighborhood of $\Theta_{q,\widetilde q,w}(p)$ such that $\det DH_{v,w}(z)\neq 0$ for every $z \in U'_{v,w}$. We conclude that a neighborhood $V_{q,\widetilde q, v} \subset V_{q,\widetilde q}$ of $p$ such that the map $\Theta_{q,\widetilde q,v}:V_{q,\widetilde q, v} \to S_qN \times \R$ is a smooth coordinate map, can be defined for instance by formula
$$
V_{q,\widetilde q, v}:=\Theta_{q,\widetilde q, v}^{-1}(U'_{v,w}).
$$
\color{black}

Finally we will prove the equations  \eqref{eq:interior_coordinates_matc} and  \eqref{eq:determinants_interior_match}. Let $p \in M_1$. By \eqref{eq:geodesics_are_the_same} and \eqref{eq:the_set_of_good_dir_are_the_same}  we have that $(q,\eta),(\widetilde q, \widetilde \eta) \in K_{G}(p)$ and \eqref{eq:cond_for_q_tilde} holds if and only if the same is true for $(\phi (q),D \Phi \eta),(\phi(\widetilde q), D\Phi \widetilde \eta)$. If  $(q,\eta),(\widetilde q, \widetilde \eta) \in K_{G}(p)$ and \eqref{eq:cond_for_q_tilde} holds then by \eqref{eq:Phi_preserves_boundary_inner_product} and \eqref{eq:coordinates_match} the equation \eqref{eq:interior_coordinates_matc} follows. 

Lastly we notice that $B_{\phi(\widetilde q)}=D \Phi (B_{\widetilde q})$. Therefore, we conclude that for given $(q,\eta),(\widetilde q, \widetilde \eta) \in K_{G}(p)$ for which \eqref{eq:cond_for_q_tilde} is valid the equation \eqref{eq:necessary_cond_for_v} holds for $v \in T_{\widetilde q}N_1$ if and only if the same holds for  $D \Phi v \in T_{\phi (\widetilde q)}N_2$. Therefore, the equation \eqref{eq:determinants_interior_match} follows from \eqref{eq:interior_coordinates_matc}. 
\color{black}
\end{proof}

\color{black}
\subsection{Boundary coordinates}

Next we construct the coordinates near $\p M$. Let $p \in \p M$. Let $I$ be the set  of self-intersection directions of $p$ (see \eqref{eq:self_intersection_directions}). By Lemma \ref{Le:non_intersection_directions} the set $S_pN \setminus I$ is open and dense in $S_pN$. As $p \in \p M$ the data \eqref{bigger geodesic measurment data} determines the set $I$, since it holds that
$$
I=\{(p,\xi) \in \p_-S_pM: \hbox{ there exists } \eta \in \p_-S_pM, \: \eta \neq \xi \hbox{ such that } \overline \Sigma(p,\xi)=\overline \Sigma(p,\eta)\}.
$$
 
Note that the usual boundary normal coordinates might not work well with our data, since it might happen that $p$ is a self-intersecting point of the boundary normal geodesic $\gamma_{p,\nu}$. In this case it is difficult to reconstruct the map $U \ni w \mapsto z(w) \in \p M$, from the data \eqref{bigger geodesic measurment data}. Here $U$ is the domain of boundary normal coordinates and $z(w)$ is the closest boundary point. To remedy this issue, we will prove in the next lemma that any non-vanishing, non-tangential, inward pointing vectorfield on $\p M$ determines a similar coordinate system as the boundary normal coordinates.

\begin{Lemma}
\label{Le:coordinates_by_F_V}
Let $(\overline{M}, g)$ be a smooth compact manifold with strictly convex  boundary and let $W$ be a non-vanishing, non-tangential and inward pointing smooth vector field on $\p M$. Let $p \in \p M$.  Then there exists $T>0$  such that the map
\begin{equation}
\label{eq:the_diffeo_F_V}
E_W: \p M  \times [0,T) \to \overline{M}, \: E_W(z,t)=\exp_{z}(tW(z))
\end{equation}
is well defined and moreover, there exists $T_p \in (0,T)$ and $\epsilon>0$ such that the restriction 
$$
E_W: [0,T_p) \times B_{\p M}(p,\epsilon) \to \overline{M}
$$
is a diffeomorphism. Here 
$$
B_{\p M}(p,\epsilon):=\{z \in \p M: d_{\p M}(p,z) < \epsilon\}
$$
and $d_{\p M}(p,z)$ is the distance from $p$ to $z$ along the boundary. 
\end{Lemma}
\begin{proof}
Since $W(z) \in (\p_+ SM)^{int}$ for every $z \in \p M$ and  $\p M$ is compact, the Lemma \ref{Le:exit_time_func_is_cont}  guarantees that there exists $T>0$ such that for every $t \in [0,T]$ and $z \in \p M$ it holds that $E_W(t,p)\in \overline{M}$.

Let $(z^j)_{j=1}^{n-1}$ be local coordinates at $p$ on the boundary.  Since the map $E_W$ is smooth it suffices to show that the Jacobian matrix $DE_W$ is invertible at $(p,0)$. 
By straightforward computation we have
\begin{equation}
\label{eq:rows_of_JF_V}
\frac{\p }{\p t}E_W(p,t)\bigg|_{t=0}=W(p) \hbox{ and } \frac{\p }{\p z^j}E_W(z,0)\bigg|_{z=p}=  \frac{\p }{\p z^j}, \: j \in \{1,\ldots, n-1\}.
\end{equation}
Since $W(p)$ is not tangential to the boundary, it holds by \eqref{eq:rows_of_JF_V} that $DE_W$ at $(p,0)$ is invertible.
\end{proof}

Choose a non-vanishing, non-tangential inward pointing smooth vector field $W$ on $\p M$ such that $W(p)\notin I=I_p$. 
By Lemma \ref{Le:coordinates_by_F_V} there exists such a neighborhood $V$ of $p$ such that the map  $E_W^{-1}:V \to \p M \times [0,T_p)$ is a smooth coordinate map. Let $w \in V$. We denote 
$$
\Pi_{W}(w):=z\in \p M \hbox{ if and only if $E_W(z,t)=w$ for some unique } t\in[0,T_p).
$$
See Figure \eqref{fig:self_intersect}. We note that the map $\Pi_W$ is smooth.

\begin{figure}[h]
 \begin{picture}(100,150)
  \put(-95,-10){\includegraphics[width=300pt]{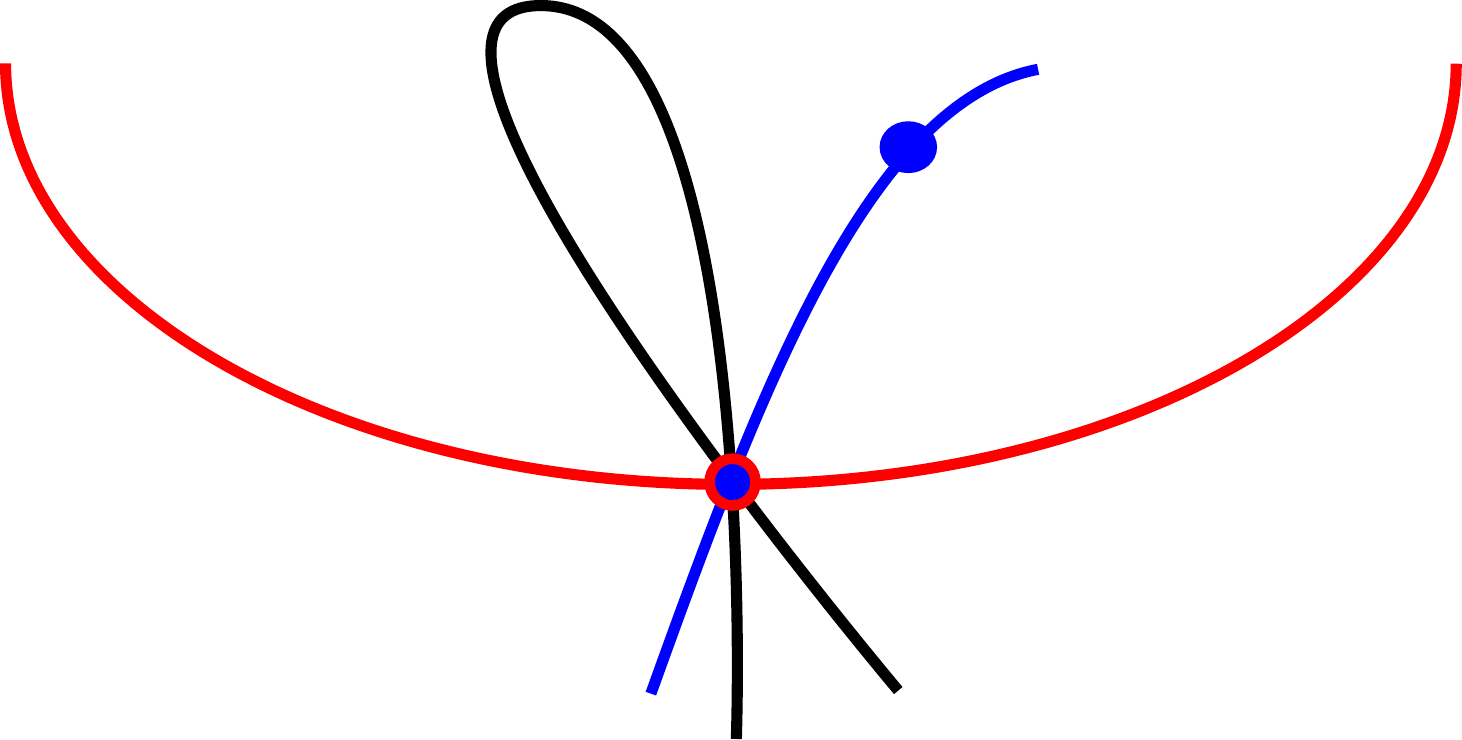}}
  \put(100,105){$w$}
  \put(120,125){$\gamma_{p,W(p)}$}
  \put(-15,140){$\gamma_{p,\nu}$}
  \put(-35,40){$\p M$}
  \put(42,33){$p$}
   \end{picture}
\caption{Here is a schematic picture about the situation where the boundary normal geodesic $\gamma_{p,\nu}$ (the black curve) is self-intersecting at $p \in \p M$ (red\&blue dot). Here the blue curve is the geodesic $\gamma_{p,W(p)}$, where $W(p)\notin I_p$. For the point $w \in M$ (blue dot) the point $p$ satisfies $p=\Pi_W(w)$.}
\label{fig:self_intersect}
\end{figure}
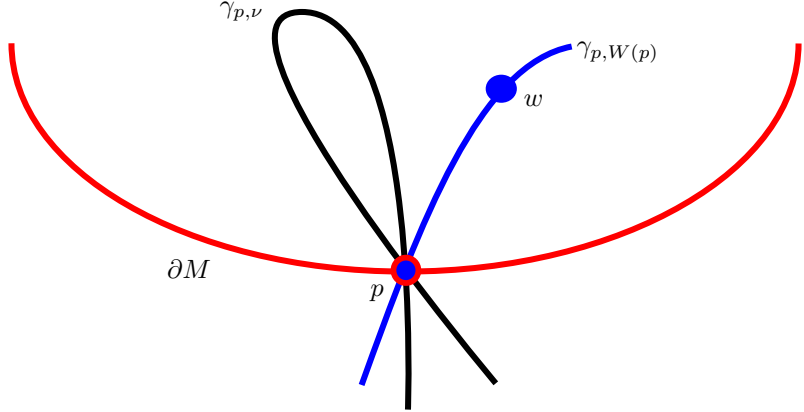

In the next Lemma we show that the data $\eqref{bigger geodesic measurment data}$ determines a continuous map $\Pi_W^E$ that coincides with $\Pi_W$ near $p$. Later we will use the map $\Pi_W^E$ to construct $(n-1)$-coordinates for points near $p$.  
\begin{Lemma}
\label{Le:good_neigborhood_for_boundary_point}
Let $(N,g)$ and $M$ be as in the Theorem \ref{th:main}. For given $r \in R_{\p M}^E( \p M)$ where $p\in \p M$ is the unique point for which $R_{\p M}^E(p)=r$ and let $W$ be a non-vanishing, non-tangential inward pointing smooth vector field such that $W(p)\notin I$,  the data \eqref{bigger geodesic measurment data} with $r$ and $W$,  determine such a $(q,\eta) \in K_G(p)$ and a neighborhood $V_{q,W} \subset \overline M$ of $p$, $q \notin V_{q,W}$ for which the map $\Theta_q:V_{q,W} \to S_qN$ is well defined. Moreover, 

\begin{equation}
\label{eq:defining_property_for_good_neigborhood_for_boundary_point2}
\dot \gamma_{q,-\Theta_q(p)}(\tau_{exit}(q,-\Theta_q(p))) \nparallel W(p)
\end{equation}
and
\begin{equation}
\label{eq:orthogonal_projection_to_boundary}
\hbox{a map }\Pi^E_{q,W}:V_{q,W}\to \p M\cap V_{q,W}, \: \hbox{ is determined by \eqref{bigger geodesic measurment data}}.
\end{equation}
Here $\Pi^E_{q,W}$ is a continuous map which coincides with $\Pi_W$ near $p$.

More precisely, let $(N_i,g_i)$ and $M_i$ be as in Theorem \ref{th:main} and suppose that \eqref{eq:collar_neib} and \eqref{eq:equivalence_of_biger_data} hold. 
Let $W$ be a smooth non-vanishing, non-tangential, inward pointing vector field on $\p M_1$ such that $W(p)\notin I$.
If 
\begin{itemize}
\item[(i)] $r=R^{E}_{\p M_1}(p)\in R^{E}_{\p M_1}(\p M_1)$
\item[(ii)] $(q,\eta) \in K_G(p)$ 
\item[(iii)] $V_{q,W}$ is neighborhood of $p$ such that $q \notin V_{q,W}$ and for every $z \in V_{q,W}$ the property \eqref{eq:coordinates_match} holds.
\end{itemize}
are given. Then 
\begin{equation}
\label{eq:defining_property_for_good_neigborhood_for_boundary_point3}
\begin{array}{ll}
& \dot \gamma_{q,-\Theta_q(p)}(\tau_{exit}(q,-\Theta_q(p))) \nparallel W(p)
\\
\hbox{if and only if}&
\\
& \dot \gamma_{\phi (q),-\Theta_{\phi(q)}(\phi(p)))}(\tau_{exit}(\phi (q),-\Theta_{\phi (q)}(\phi(p)))) \nparallel D\Phi W(\phi(p)),
\end{array}
\end{equation}
and
\begin{eqnarray}
\label{eq:the_boundary_points_are_the_same}
\phi ( \Pi^E_{q,W}(z))= \Pi^E_{\phi(q), D \Phi W}(\Psi(z)), \: \hbox{ for all } z \in V_{q,W}.
\end{eqnarray}
\end{Lemma}
\begin{proof}
We start by showing how property \eqref{eq:defining_property_for_good_neigborhood_for_boundary_point2} can be verified. Let $R_{\p M}^E(U) \subset R_{\p M}^E(\overline{M})$ be a neighborhood of $r=R_{\p M}^E(p)$. For every  $r' \in (R_{\p M}^E(U)  \cap R_{\p M}^E(\p M)), \: r' \neq r$ we define the number $N(r')$ to be the cardinality of the set 
$$
A(r,r'):=\{\xi \in (S_zN) \cap r; \: \overline{\Sigma}(z,\xi) \subset R^{E}_{\p M}(U), \: (z,\xi) \in K_{G}(p)\},
$$
here $z=(R^E_{\p M})^{-1}(r')$. By Lemma \ref{Le:minimizing_geo_from_boundary} there exists a neighborhood $R_{\p M}^E(U') \subset R_{\p M}^E(U) $ of $r$ such that $N(r')=1$ for every $ r' \in (R_{\p M}^E(U')  \cap R_{\p M}^E(\p M)), \: r'\neq r$. Since the sets $A(r,r')$ are determined by $r$ and \eqref{bigger geodesic measurment data} we can find the set $R_{\p M}^E(U')$. We denote by  $U' \subset U$ the unique neighborhoods of $p$ that are related to $R_{\p M}^E(U')$ and  $R_{\p M}^E(U)$. We write $z=(R^{E}_{\p M})^{-1}(r') \in U'$, for given $r' \in R_{\p M}^E(U')$. Therefore, by Lemma \ref{Le:n-1_coordinates} for the vector $\xi \in A(r,r')$, there exists a neighborhood $V_z\subset U'$ of $p$ such that $z \notin V_z$ and the map $\Theta_z:V_z \to S_zN$ is smooth and well defined. Moreover there exists a unique $(p,v_z)\in S_pN$ so that 
$$
\gamma_{p,v_z}(\tau_{exit}(p,v_z))=z \hbox{ and } \dot \gamma_{p,v_z}(\tau_{exit}(p,v_z))=\Theta_z(p)\in A(r,r')
$$
and
$$
\overline{\Sigma}(z,\Theta_z(p))=\overline{\Sigma}(p,v_z).
$$
Then $v_z= \dot \gamma_{z,-\Theta_z(p)}(\tau_{exit}(z,-\Theta_z(p)))$ and thus the data \eqref{bigger geodesic measurment data} determines, if $v_z$ and $W(p)$ are parallel. It follows from Lemma \ref{Le:coordinates_by_F_V} that there exists  $q \in U'$ such that $v_q$ is not parallel to $W(p)$. Therefore we can use data \eqref{bigger geodesic measurment data} to determine a point $q$ that satisfies the property  \eqref{eq:defining_property_for_good_neigborhood_for_boundary_point2}.
Let $(q,\eta) \in A(r,R_{\p M}^E(q))$. Write $V_q\subset U'$ for a neighborhood of $p$, $q \notin V_q $ such that the map $\Theta_q: V_q \to S_qN$ is well defined. 


\medskip

Next we will construct the map $\Pi_W^E$, see \eqref{eq:orthogonal_projection_to_boundary}. Since $\gamma_{p,-W}$ is not self-intersecting at $p$ there exists $T>0$ such that for any $t\in [0,T]$ the point $\gamma_{p,-W}(t)$ is not a self-intersecting point on geodesic $\gamma_{p,-W}$. We will prove that there exists $0<T'\leq T$ and $\epsilon>0$ such that for any $t \in [0,T']$ and $z \in B_{\p M}(p,\epsilon)$ the point $\gamma_{z,-W}(t)$ is not a self-intersecting point of the geodesic $\gamma_{z,-W}$. 

If that is not true, then there exists a sequence $(p_j)_{j=1}^\infty \subset \p M$ that converges to $p$ as $j \to \infty$ and moreover, there exist sequences $(t_j)_{j=1}^\infty, (T_j)_{j=1}^\infty  \subset \R$ such that for every $j \in \N$, we have $T_j>t_j\geq 0$, $T_j\leq \tau_{exit}(p_j,-W)$, 
$$
\gamma_{p_j,-W}(t_j)=\gamma_{p_j,-W}(T_j) \hbox{ and } t_j \to 0 \hbox{ as } j \to \infty.
$$
By Lemma \ref{Le:exit_time_func_is_cont} we may assume that $T_j \to T_0\in \R_+$ as $j \to \infty$. Then
$$
p=\lim_{j \to \infty} \gamma_{p_j,-W}(t_j)=\lim_{j \to \infty} \gamma_{p_j,-W}(T_j)=\gamma_{p,-W}(T_0).
$$
Notice that $T_0>0$ since otherwice geodesic loops $\gamma_{p_j,-W}([t_j,T_j])$ would be short and this would contradict Lemma \ref{Le:minimizing_geo_from_boundary}. Therefore due to Corollary \ref{Co:interior_of_geo} it must hold that $T_0=\tau_{exit}(p,-W)$. But this contradicts the assumption that $p$ is not a self-intersecting point of $\gamma_{p,-W}$.

Consider the set
$$
O(p,T^{(1)},\epsilon^{(1)}):=\{\exp_z(sW(z))\in \overline{M}; \: z \in B_{\p M}(p,\epsilon^{(1)}), \: s \in(0,T^{(1)})\} \subset V_q,
$$
where $T^{(1)}>0$ and $\epsilon^{(1)}>0$  are so small that  for any $t \in [0,T^{(1)}]$ and $z \in B_{\p M}(p,\epsilon^{(1)})$ the point $\gamma_{z,-W}(t)$ is not a self-intersecting point on the geodesic $\gamma_{z,-W}$.  For every $w \in \p M \cap O(p,T^{(1)},\epsilon^{(1)})$ we denote the component of $\overline{\Sigma}(w,W) \cap R_{\p M}^E(O(p,T^{(1)},\epsilon^{(1)}))$ that contains $R_{\p M}^{E}(w)$ by $C(w,T^{(1)},\epsilon^{(1)})$. Moreover, due to the choice of $T^{(1)}$ and $\epsilon^{(1)}$ it holds that $R_{\p M}^{E}(\gamma_{w,W}([0,T^{(1)}])) = C(w,T^{(1)},\epsilon^{(1)})$, for every $w \in \p M \cap O(p,T^{(1)},\epsilon^{(1)})$.  

By Lemma \ref{Le:coordinates_by_F_V}  there exists $T^{(2)} \in (0, T^{(1)})$ and $\epsilon^{(2)}\in (0,\epsilon^{(1)})$ such that, the set $O(p,T^{(2)},\epsilon^{(2)})$ is open and for all  $w,z \in (\p M \cap O(p,T^{(2)},\epsilon^{(2)})), \: w\neq z$ it holds that
$$
\gamma_{w,W}([0,T^{(2)}]) \cap \gamma_{z,W}([0,T^{(2)}])=\emptyset.
$$
Therefore
$$
C(w,T^{(2)},\epsilon^{(2)}) \cap C(z,T^{(2)},\epsilon^{(2)})=\emptyset.
$$
Finally we notice that 
$$
R_{\p M}^E(O(p,T^{(2)},\epsilon^{(2)}))=\bigcup_{w \in (\p M \cap O(p,T^{(2)},\epsilon^{(2)}))}C(w,T^{(2)},\epsilon^{(2)}).
$$
Then for a point $z \in O(p,T^{(2)},\epsilon^{(2)})$ there exists a unique $w \in (\p M \cap O(p,T^{(2)},\epsilon^{(2)}))$ such that $R_{\p M}^E(z) \in C(w,T,^{(2)}, \epsilon^{(2)})$.

Next we will show how we can determine a set $V_{q,W} \subset V_q$ that is similar to $O(p,T^{(2)},\epsilon^{(2)})$ just using the data \eqref{bigger geodesic measurment data}. Let $\epsilon >0$ be so small that $B_{\p M}(p,\epsilon) \subset V_q$. For every $w \in B_{\p M}(p,\epsilon)$ we write $C(w)$ for the maximal connected subset of $\overline{\Sigma}(w,W)$ that contains $R_{\p M}^E(w)$ and satisfies 
\begin{itemize}
\item[(P1)] The sets $C(w)$ are homeomorphic to $[0,1)$.
\item[(P2)] For all $w,z \in B_{\p M}(p,\epsilon), \: w\neq z$ the sets $C(w)$ and $C(z)$ do not intersect.
\item[(P3)] The set $\widetilde V:=\bigcup_{w \in B_{\p M}(p,\epsilon)} C(w) \subset R_{\p M}^E(\overline{M})$ is open and contained in $R_{\p M}^E(V_q)$.
\end{itemize}
We emphasize that properties (P1)--(P3) hold by Lemma \ref{Le:coordinates_by_F_V}, if $\epsilon>0$ is small enough. Moreover since the map $R_{\p M}^E$ is a homeomorphism, for any given value of $\epsilon$ we can verify,  by using data \eqref{bigger geodesic measurment data}, if conditions (P1)--(P3) ave valid. Thus, let us choose  $\epsilon>0$ so that (P1)--(P3) are valid and define $V_{q,W}:=(R^{E}_{\p M})^{-1}(\widetilde V)$. 

Finally we are ready to define the map $\Pi^E_{q,W}$ by the equation
\begin{equation}
\label{eq:char_of_map_Pi_W}
\Pi^E_{q,W}:V_{q,W} \to \p M, \: \Pi^E_{q,W}(z):=\{w \in (\p M \cap V_{q,W}) ; \: R_{\p M}^E(z) \in C(w)\}.
\end{equation}
By the continuity of the map $E_W$ (as in Lemma \ref{Le:coordinates_by_F_V}) and the properties (P1)--(P3) it holds that the map $\Pi^E_{q,W}$ is continuous on $V_{q,W}$. Moreover due to Lemma \ref{Le:coordinates_by_F_V} the map $\Pi^E_{q,W}$ coincides with $\Pi_{W}$ near the point $p$.

\medskip
Let $p \in \p M_1$ and let vector field $W \in (\p_+SM)^{int}$ be such that $W(p)\notin I$.  Choose $(q,\eta) \in K_{G}(p)$ and $V_{q,W}$, a neighborhood of $p$ such that $q \notin V_{q,W}$ for which \eqref{eq:coordinates_match} holds.  Finally we will verify that properties \eqref{eq:defining_property_for_good_neigborhood_for_boundary_point3}--\eqref{eq:the_boundary_points_are_the_same}  are valid for $p$ if and only if they are valid for $\Psi(V_{q,W})\subset \overline{M}_2$.  

The property \eqref{eq:defining_property_for_good_neigborhood_for_boundary_point3} follows from  \eqref{eq:Phi_preserves_boundary_inner_product} and \eqref{eq:geodesics_are_the_same}.

By the definiton \eqref{eq:the_map} of the mapping $\Psi$ the properties (P1)--(P3) are valid in $V_{q,W}$ if and only if the same holds for $\Psi(V_{q,W})$. Therefore \eqref{eq:the_boundary_points_are_the_same} follows from \eqref{eq:char_of_map_Pi_W}.
\end{proof}

Given $p\in \p M$ and a non-vanishing, non-tangential inward pointing smooth vector field $W$ on $\p M$ such that $W(p)\notin I$, we choose $(q,\eta) \in K_G(p)$ and a neighborhood $V_{q,W}$ of $p$ such that $q \notin V_{q,W}$ as in Lemma \ref{Le:good_neigborhood_for_boundary_point}. Let $v \in T_qN$ and define 
\begin{equation}
\label{eq:the_map_Q_{q,v}}
Q_{q,v}: V_{q,W} \to \R, \:Q_{q,v}(x):=\langle v,\Theta_q(x) \rangle_g.
\end{equation} 
Now we are ready to construct a boundary coordinate system near $p\in \p M$.
%
\begin{Lemma}
\label{Le:smooth_boundary_atlas}
Let $(N,g)$ and $M$ be as in the Theorem \ref{th:main}. Let $r \in R_{\p M}^E( \p M)$ and let $p\in \p M$ to be the unique point for which $R_{\p M}^E(p)=r$.  Let $W$ be a non-vanishing, non-tangential inward pointing smooth vector field on $\p M$ such that $W(p)\notin I$, we choose $(q,\eta) \in K_G(p)$ and a neighborhood $V_{q,W}$ of $p$ such that $q \notin V_q$ as in Lemma \ref{Le:good_neigborhood_for_boundary_point}. Then there exists $v \in T_qN$ and a neighborhood $V_{q,v,W}$ of $p$ such that $(\widetilde Q_{q,v}, \Pi^E_{q,W}):V_{q,v,W} \to (\R \times \p M)$ defines a smooth boundary coordinate system near $p$. Here
$$
\widetilde Q_{q,v}(x):= \pm (Q_{q,v}(\Pi^E_{q, W}(x))-Q_{q,v}(x)),
$$
where the sign is chosen such that $\widetilde Q_{q,v}\geq 0$. See Figure \eqref{fig:bd_coord}.


Moreover, for given $r=R_{\p M}^E(p) \in R_{\p M}^E(M)$, $W$ such that $W(p)\notin I$ and $(q,\eta) \in K_G(p)$ the data \eqref{bigger geodesic measurment data} determines $v \in T_qN$ and a neighborhood $R_{\p M}^E(V_{q,v,W})$ of $r$ and the map $(\widetilde Q_{q,v}, \Pi^E_{q,W})\circ (R_{\p M}^E)^{-1}:R_{\p M}^E(V_{q,v,W}) \to (\R \times \p M)$, such that $(\widetilde Q_{q,v}, \Pi^E_{q,W}):V_{q,v,W} \to (\R \times \p M)$ is a smooth coordinate map from a neighborhood $V_{q,v, W}$ of $p$ onto some open set $U \subset (\R \times \p M)$.

\color{black}

More precisely, let $(N_i,g_i)$ and $M_i$ be as in Theorem \ref{th:main} such that \eqref{eq:collar_neib} and \eqref{eq:equivalence_of_biger_data} hold. Then for a given $p \in \p M_1$,  $(q,\eta)\in K_G(p)$, $W$, that is a non-vanishing, non-tangential inward pointing smooth vector field on $\p M_1$ such that $W(p)\notin I$, neighborhood $V_{q,W}$ of $p$ as in Lemma \ref{Le:good_neigborhood_for_boundary_point} and $v\in T_{q}N_1$  we have that \eqref{eq:the_boundary_points_are_the_same} holds. If $s\in \R$, $w \in \p M_1$ and $z \in V_q$ then
\begin{eqnarray}
\label{eq:the_boundary_angles_are_the_same}
\widetilde Q_{q,v}(z)=(s,w) \hbox{ if and only if } \widetilde Q_{\phi(q),D\Phi v}(\Psi(z))=(s,\phi(w)),
\end{eqnarray}
and
\begin{equation}
\label{eq:determinants_boundary_match}
\det D (\widetilde Q_{q,v}, \Pi^E_{q, W})(p)\neq 0, \hbox{ if and only } \det D(\widetilde Q_{\phi(q),D\Phi v}, \Pi^E_{\phi(q), D \Phi W})(\Psi(p)) \neq 0.
\end{equation}
\end{Lemma}

\begin{figure}[h]
 \begin{picture}(100,150)
  \put(-100,20){\includegraphics[width=300pt]{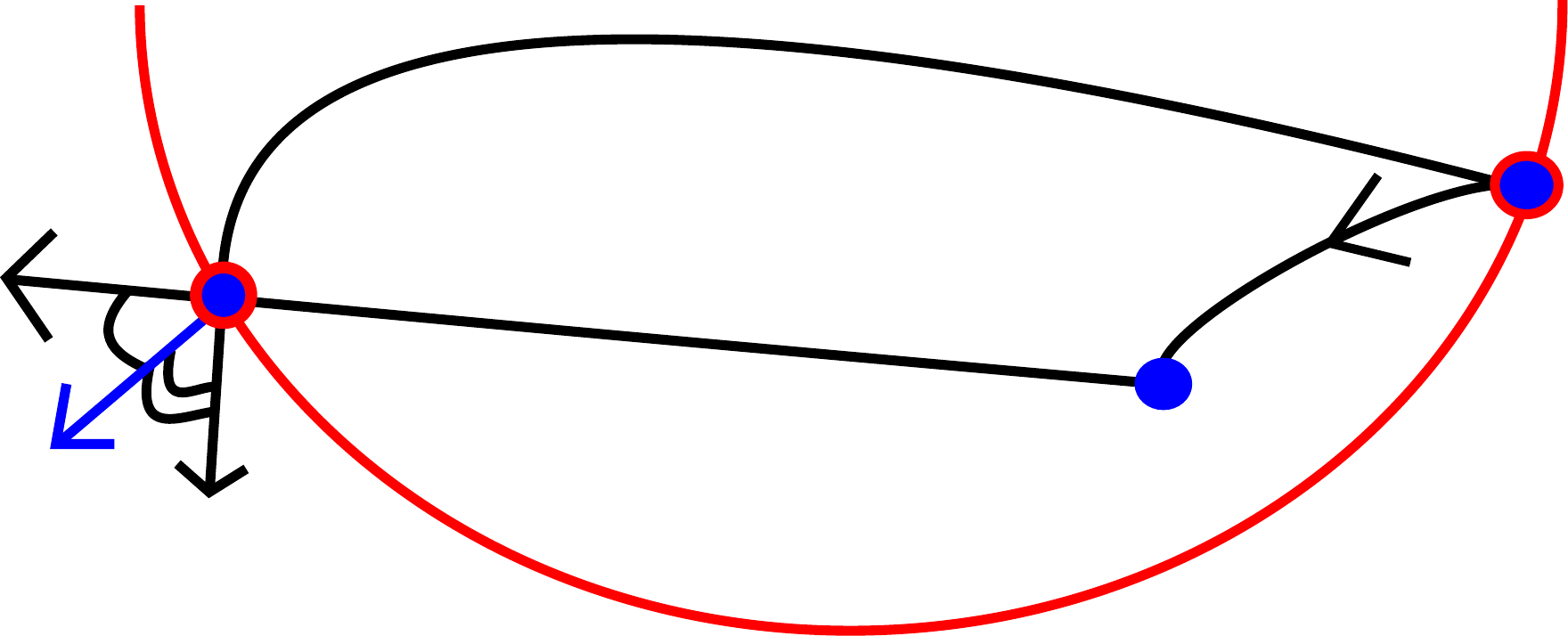}}
  \put(20,10){$\p M$}
  \put(120,55){$x$}
  \put(205,103){$\Pi^E_{q,W}(x)$}
  \put(-75,35){$\Theta_{q}(\Pi^E_{q,W}(x))$}
  \put(-90,45){$v$}
  \put(-110,103){$\Theta_{q}(x)$}
  \put(-50,90){$q$}
   \end{picture}
\caption{Here is a schematic picture about the map $(\widetilde Q_{q,v}, \Pi^E_{q, W})$ evaluated at a point $x\in M$ that is close to $\p M$, where the point $x \in M$ is the blue dot. The right hand side red\&blue dot is $\Pi^E_{q,W}(x)$ and the left hand side red\&blue dot is $q$. The blue arrow is the given vector $v \in T_qN$.}
\label{fig:bd_coord}
\end{figure}

\begin{proof}
Recall that by Lemma \ref{Le:ker_of_Dtheta} and \eqref{eq:defining_property_for_good_neigborhood_for_boundary_point2} it holds that $W(p) \;$is not contained in ker$(D \Theta_q)$. Therefore there exists $v \in T_qN$ such that
$$
\langle v, D\Theta_qW(p) \rangle_g \neq 0.
$$ 
Choose the local coordinates at $p$ with the map $E_W$ (as in Lemma \ref{Le:coordinates_by_F_V}). Then with respect to these coordinates we have
$$
D(\widetilde Q_{q,v}, \Pi^E_{q, W})(p)=\left(\begin{array}{cc}
\overline 0^T & -(W Q_{q,v})(p)
\\
Id_{n-1} & \overline 0
\end{array} \right),
$$
since $\Pi^E_{q,W}$ coincides with $\Pi_W$ near $p$. Also it holds that
\begin{equation}
\label{eq:WQ_q_v}
(W Q_{q,v})(p)=\langle v, D\Theta_q W(p) \rangle_g \neq 0.
\end{equation}
Therefore $D(\widetilde Q_{q,v}, \Pi^E_{q, W})(p)$ is invertible. Moreover, $(\widetilde Q_{q,v}, \Pi^E_{q, W})(z)=(0,z)$ if $z \in \p M$ and due to  \eqref{eq:WQ_q_v}, Lemma \ref{Le:coordinates_by_F_V} and the Taylor expansion of $\widetilde Q(q,v)$ it holds that $\widetilde Q_{q,v}(w)\neq 0$ when $w \notin \p M\cap V_{q,W}$ is close to $p$. 

We choose so small neighborhood $V_{q, v, W}$ for $p$ such that $\widetilde Q_{q,v}: V_{q, v, W}\to \R$ vanishes only in $V_{q, v, W} \cap \p M$. Therefore we have proven that $(\widetilde Q_{q,v}, \Pi^E_{q, W}): V_{q, v, W}\to \R^n$ is a smooth boundary coordinate map at $p$.

\color{black}

\medskip
Notice that for given $r \in R_{\p M}^E(\p M)$, $W$, $(q,\eta), \in K_G(p)$ and $ v\in T_{q}N$ the data \eqref{bigger geodesic measurment data} determines a neighborhood $R^E_{\p M}(V_{q, v, W})$ of $r$ and the map $(\widetilde Q_{q,v}, \Pi_{q, W}) \circ (R^E_{\p M})^{-1}$ by a similar argument as in the proof of Lemma \ref{Le:smooth_atlas}.

\medskip
The equations  \eqref{eq:the_boundary_angles_are_the_same} and \eqref{eq:determinants_boundary_match} can be proven by a similar argument as in the proof of Lemma \ref{Le:smooth_atlas}.
\end{proof}

\medskip
Let $A$ be such an index set that collections $(V_{q,\widetilde q, v}^\alpha,\Theta_{q,\widetilde q,v}^\alpha)_{\alpha \in A}$ and
\\
 $(V^\beta_{q,v,W},(\widetilde Q_{q,v}, \Pi^E_{q,W})^\beta)_{\beta \in A}$, that are as in Lemmas \ref{Le:smooth_atlas} and \ref{Le:smooth_boundary_atlas}, form a smooth atlas of $(\overline M, g)$. 

We define an atlas for $R^{E}_{\p M}(\overline M)$ using the following charts
\begin{equation}
(R^E_{\p M}(V_{q,\widetilde q, v}^\alpha),\Theta_{q,\widetilde q,v}^\alpha\circ (R^E_{\p M})^{-1})_{\alpha \in A} \hbox{ and } (R^E_{\p M}(V^\beta_{q,v,W}),(\widetilde Q_{q,v}, \Pi^E_{q,W})^\beta \circ (R^E_{\p M})^{-1})_{\beta \in A}.
\label{eq:smooth_atlas_for_data}
\end{equation}

By Lemmas   \ref{Le:smooth_atlas} and \ref{Le:smooth_boundary_atlas} the charts given in \eqref{eq:smooth_atlas_for_data} define a smooth structure for $R^{E}_{\p M}(\overline M)$. 
\begin{Proposition}
\label{Pr:reconstruction_of_smooth_struct}
Let $(\overline{M}, g)$ be as in Theorem \ref{th:main}. If data \eqref{bigger geodesic measurment data} is given, we can find the smooth charts of \eqref{eq:smooth_atlas_for_data} using only the data  \eqref{bigger geodesic measurment data}. Moreover with respect to this smooth structure the map $R^E_{\p M}:\overline{M}\to R^E_{\p M}(\overline{M})$  is a diffeomorphism.
\end{Proposition}
\begin{proof}
The claim follows from Lemmas  \ref{Le:smooth_atlas}, \ref{Le:smooth_boundary_atlas} and the definition \eqref{eq:smooth_atlas_for_data}. 
\end{proof}
\color{black}
\medskip
Now we are ready to present the main theorem of this section.
\begin{Theorem}
\label{th:Psi_is_diffeo}
The mapping $\Psi:\overline{M}_1 \to \overline{M}_2$ is a diffeomorphism.
\end{Theorem}
\begin{proof}
Since  maps $R_{\p M_i}^E$, $i \in \{1,2\}$ are diffeomorphisms it suffices to prove that 
$$
\widehat{D \Phi}: R_{\p M_1}^E (\overline M_1)\to R_{\p M_2}^E(\overline M_2)
$$
is a diffeomorphism. 

Let $p \in M_1$. By Lemma \ref{Le:smooth_atlas} there exist $(q,\eta), (\widetilde q, \widetilde \eta) \in K_G(p)$, $v \in T_{\widetilde q}M_1$ and a neighborhood $V=V_{q,\widetilde q, v}$ of $p$ such that the map $\Theta_{q,\widetilde q, v}:V \to (S_qM_1 \times \R)$ is a smooth coordinate map. Then it follows from \eqref{eq:interior_coordinates_matc} and \eqref{eq:determinants_interior_match} that $\Theta_{\phi(q),\widetilde \phi(q), D\Phi v}:\Psi(V) \to (S_qM_2 \times (\R\setminus \{0\}))$ is a smooth coordinate map. Therefore the local representation of $\widehat{D \Phi}$, with respect to the smooth structures as in \eqref{eq:smooth_atlas_for_data}, on $R^E_{\p M}(V)$ is given by
\begin{equation}
\label{eq:local_rep_of_Psi_1}
(\p_-S_{q}\overline{M}_1 \times  (\R\setminus \{0\}))\ni(\xi,s) \mapsto (D\Phi \, \xi,s) \in (\p_- S_{\phi(q)}\overline{M}_2 \times  (\R\setminus \{0\})).
\end{equation}
Since $\Phi:K_1 \to K_2$ (see \eqref{eq:big_phi}) is a diffeomorphism it holds that  $\widehat{D \Phi}$ is  diffeomorphic on $R^E_{\p M_1}(V)$.

Let $p \in \p M_1$. By Lemma \ref{Le:smooth_boundary_atlas} there exist $(q,\eta) \in K_G(p)$, $v \in T_qM_1$, a smooth vector field $W$ on $\p M_1$, and a neigborhood $U=V_{q,v,W}$ of $p$ such that the map $(\widetilde Q_{q,v}, \Pi^E_{q, W}) :U \to \R \times \p M_1$ defines smooth local boundary coondinates. Moreover by \eqref{eq:the_boundary_angles_are_the_same} and \eqref{eq:determinants_boundary_match} the map $(\widetilde Q_{\phi(q), D\Phi v}, \Pi^E_{\phi(q), D\Phi W}) :\Psi(U) \to \R \times \p M_2$ defines smooth local boundary coordinates near $\phi(p)$.  Therefore the  local representation of $\widehat{D \Phi}$ on $R^E_{\p M_1}(U)$, with respect to the smooth structures as in \eqref{eq:smooth_atlas_for_data},  is given by 
\begin{equation}
\label{eq:local_rep_of_Psi_2}
(\R \times \p M_1)\ni (s,z)\mapsto (s,\phi(z)) \in \R \times \p M_2  .
\end{equation}
Since $\phi:\p M_1 \to \p M_2$ is a diffeomorphims it holds that  $\widehat{D \Phi}$ on $R^E_{\p M_1}(U)$ is diffeomorphic to its image.

Thus we have proved that $\Psi$ is a local diffeomorphism. By Theorem \ref{th:topology} the map $\Psi$ is one-to-one and therefore it is a global diffeomorphism.

\end{proof}

\color{black}

\section{The reconstruction of the Riemannian metric}
So far we have shown that the map $\Psi:\overline{M}_1 \to \overline{M}_2$ is a diffeomorphism. The aim of this section is to show that $\Psi$ is a Riemannian isometry. We start with showing that $\Psi$ preserves the boundary metric. This is formulated precisely in the following Lemma.
\begin{Lemma}
\label{Le:Psi_pulls_back_the_boundary_metric}
The map $\Psi:\overline{M_1} \to \overline{M}_2$ has a property
$$
(\Psi^\ast g_2)|_{\p M_1}=g_1|_{\p M_1}. 
$$
\end{Lemma}
\begin{proof} 
Since $\Psi$ is smooth, the metric tensor $\Psi^\ast g_2$ is well defined.  By  Theorem \ref{th:topology} and \eqref{eq:collar_neib} we have
$$
(\Psi^\ast g_2)|_{\p M_1}
=\phi^{\ast}(g_2|_{\p M_2})=g_1|_{\p M_1}.
$$
\end{proof}

We denote $\widetilde g:=\Psi^{\ast}g_2$ the pullback metric on $\overline{M}_1$ and $g:=g_1$. We will also denote from now on $\overline{M}_1=\overline{M}$. By Lemma  \ref{Le:Psi_pulls_back_the_boundary_metric} $\widetilde g$ and $g$ coincide on the boundary $\p M$. Next we will show that they also have the same Taylor expansion at the boundary.

\begin{Proposition}
\label{le:jet}
Suppose that $N, M, g$ and  $\widetilde g$ are as in Theorem \ref{th:main} such that data \eqref{eq:equivalenc_of_data}  and \eqref{eq:collar_neib} is valid with $\phi=id$. 

Let $(x^1,\ldots, x^n)$ be any coordinate system near the boundary and $\alpha\in \N^{n}$ any multi-index. Write $g=(g_{ij})_{i,j=1}^{n}$ and $\widetilde g=(\widetilde g_{ij})_{i,j=1}^{n}$. Then for all $i,j \in \{1, \ldots,n\}$ holds
\begin{equation}
\label{eq:jet}
\partial^\alpha g_{ij} |_{\p M}= \partial^\alpha \widetilde g_{ij} |_{\p M}, \: \p^\alpha:=\prod_{k=1}^n \left(\frac{\p}{\p x^k}\right)^{\alpha_k}
\end{equation}
\end{Proposition}
\begin{proof}
Let $(p,\eta) \in  \p_- SM$. Suppose that $\eta$ is not tangential, since $M$ is non-trapping there exists a unique vector $(q,\xi) \in \p_- SM$ such that $(q,\xi)\neq (p,\eta)$ and  
$$
\overline\Sigma(p,\eta)=\overline\Sigma(q,\xi),
$$ 
where set $\overline\Sigma(p,\eta)$ is defined as in \eqref{eq:geodesic_segments1}. Therefore, the data \eqref{bigger geodesic measurment data} determines the \textit{scattering relation}
$$
L_g: \p_+ SM\to   \p_- SM, \: L(p,-\eta)=(\gamma_{p,-\eta}(\tau_{exit}(p,-\eta)),\dot \gamma_{p,-\eta}(\tau_{exit}(p,-\eta)))=(q,\xi).
$$ 
Thus by \eqref{eq:image_of_closed_geo} it holds that $L_{\widetilde g}=L_g=:L$, where $L_{\widetilde g}$ is the scattering relation of metric tensor $\widetilde g$.

It is shown in \cite{stefanov2008microlocal} Proposition 2.1 that the scattering relation $L_g$ with Lemma \ref{Le:Psi_pulls_back_the_boundary_metric} determine the first exit time function $\tau_{exit}$ for $\xi$ close to $S_p\p M$. More precisely, for every $p \in M$ there exists a neighborhood $\mathcal{V}_p \subset  \p_+ S_pM $ of $S_p \p M$ such that for all $(p,\xi) \in \mathcal{V}_p$
$$
\tau_{exit}(p,\xi)=\widetilde{\tau_{exit}}(p,\xi),
$$
where $\widetilde{\tau_{exit}}$ is the first exit time function of $\widetilde g$. Denote $\mathcal{V}:=\cup_{p \in \p M}\mathcal{V}_p$.

\medskip
We call the pair $(L|_{\mathcal V}, \tau_{exit}|_{\mathcal{V}})$ the \textit{local lens data}. In Theorem 2.1. of \cite{lassas2003semiglobal} it is shown that the local lens data implies \eqref{eq:jet}.
\end{proof}

Lemma \ref{Le:Psi_pulls_back_the_boundary_metric} and Proposition \ref{le:jet} imply that we can smoothly extend $\widetilde g$ onto $N$ such that $g=\tilde g$ in $N\setminus \overline M$. Next we will show that the geodesics of $g$ and $\widetilde{g}$ are the same.


\begin{Lemma}
\label{le:geodesic_equiv}
Metrics $g$ and $\widetilde g$ are geodesically equivalent, i.e., for every geodesic $\gamma:\R \to N$ of metric $g$, there exists a reparametrisation $\alpha:\R \to \R$ such that the curve $\gamma\circ \alpha$ is a geodesic of metric $\widetilde g$, and vice versa.
\end{Lemma}
\begin{proof}
Since $M$ is non-trapping, any geodesic arc $\gamma$ of metric $g$ that is contained in $\overline M$ can be parametrized with a set $\overline \Sigma(p,\eta)$ for some $(p,\eta)\in \p_-SM =\p_- \widetilde{ S}M$. Thus by \eqref{eq:image_of_closed_geo} any geodesic $\gamma$ of $g$ can be re-parametrized to be a geodesic $\widetilde \gamma$ of metric $\widetilde g$ and vice versa. This means that there exists a smooth one-to-one and onto function  $\alpha:[a',b']\to [a,b]$ such that
$$
\dot{\alpha}>0, \: \widetilde \gamma(t)=(\gamma \circ \alpha)(t).
$$ 
See Subsection 2.0.5 of \cite{LaSa} for more details.
\end{proof}

Now we are ready to prove the main result of this section. The proof is similar to the one given in \cite{LaSa} which works for more general settings. The key ingredient of the proof is the Theorem 1 of \cite{topalov2003geodesic}.

\begin{Proposition}
The metrics $g$ and $\widetilde g$ coincide in $M$.
\label{Pr:metrics_coincide}
\end{Proposition}
\begin{proof}
Define a smooth mapping $I_0:TN \rightarrow \R$ as 
\begin{equation}
I_0(p,v)=\bigg(\frac{\det (g)|_{p} }{\det(\widetilde{g})|_{p}}\bigg)^{\frac{2}{n+1}} \widetilde{g}(v,v), \label{eq:MaTo_formula}
\end{equation}
note that the function 
$$
f(p):=\frac{\det (g)|_{p} }{\det(\widetilde{g})|_{p}} 
$$ 
is coordinate invariant, since in any smooth local coordinates $(U,(x^j)_{j=1}^n)$ the function $f$ coincides with the function $p \mapsto\det (\widetilde g^{ik}(p)g_{jk}(p))$ in $U$ where $(\widetilde g^{ij})_{i,j=1}^{n}$ is the inverse matrix of $(\widetilde g_{ij})_{i,j=1}^{n}$ and the $(1,1)$-tensor field
$$
\widetilde g^{ik}g_{jk}\frac{\p}{\p x^i}\otimes dx^j
$$ 
can be considered as a linear automorphism on $TU$.

Let $\gamma_{g}$ be a geodesic of metric $g$. Define a smooth path $\beta$ in $TN$ as $\beta(t)=(\gamma_{g}(t), \dot{\gamma}_{g}(t))$. Then $\beta$ is an integral curve of the geodesic flow of metric $g$.
Theorem 1 of \cite{topalov2003geodesic} states that, if $g$ and $\widetilde{g}$ are geodesically equivalent, 
then function $t\mapsto I_0(\beta(t))$ is a constant for all geodesics $\gamma_g$ of metric $g$.

Since $g$ and $\widetilde{g}$ coincide in the set $N \setminus M$, we have that for any  point $z \in N \setminus M$ and vector $v \in T_zN$ 
\begin{equation*}
I_0(z,v)=\widetilde{g}_z(v,v)=g_z(v,v). 
\end{equation*}
Let $p \in M$. Choose $\gamma(t):=\gamma_{z, \xi}(t),\: \xi \in S_zN,$ $z\in N \setminus M$  be a $g$-geodesic such that $p=\gamma(t_p)$ for some $t_p> 0$. Then $I_0(z,\xi)=1$ 
and  by Theorem 1 of \cite{topalov2003geodesic} $I_0$ is constant along the geodesic flow $(\gamma(t),\dot\gamma(t))$, i.e.,

\begin{equation}
I_0(p,\dot{\gamma}(t_p))=I_0(z,\xi)=1. \label{I_0=1 on W}
\end{equation}

Since $M$ is non-trapping, \eqref{I_0=1 on W} implies that, in local coordinates, for any $\xi \in S_{p}N$
\begin{equation}
g_{ij}(p)\xi^i\xi^j=1=I_0(p,\xi)=\bigg(f(p)\bigg)^{\frac{2}{n+1}} \widetilde{g}_{ij}(p)\xi^i\xi^j. \label{metrics are conformal}
\end{equation}
Then the equation \eqref{metrics are conformal} holds for all $\xi \in T_{p}N$ and both sides of equation \eqref{metrics are conformal} are smooth in $\xi$, we obtain 
\begin{equation}\label{local metrics are conformal eq}
g_{ij}(p)=\bigg(f(p)\bigg)^{\frac{2}{n+1}} \widetilde{g}_{ij}(p), \textrm{ for all } i,j\in \{1,\ldots,n\}.
\end{equation} 
Moreover
\begin{equation*}
(f(p))^{\frac{2n}{n+1}-1}=1,
\end{equation*}
which implies that $f(p)=1$. By \eqref{local metrics are conformal eq} we conclude that $g= \widetilde{g}$ at $p$. Since $p \in M$ was arbitrary the claim follows. 
\end{proof}

\begin{proof}[Proof of \ref{th:main}.]
The proof of Theorem \ref{th:main} follows from Theorems \ref{th:topology}, \ref{th:Psi_is_diffeo} and Proposition \ref{Pr:metrics_coincide}.
\end{proof}

\medskip
We will give a Riemannian metric $G$ to smooth manifold $R^{E}_{\p M}(\overline{M})$ as a pullback of $g$, that is $ G:= ((R^{E}_{\p M})^{-1})^{\ast}g$. A priori we don't know $g$, thus  we don't  know $G$. However as the smooth structure is known we can recover the collection $\Met(R^{E}_{\p M}(\overline{M}))$, that is the collection of all Riemannian metrics of the smooth manifold $R^{E}_{\p M}(\overline{M})$. 

Since $(\overline{M},g)$ is non-trapping we can use the sets $\overline \Sigma(p,\eta), \: (p,\eta) \in \p_-SM$ to recover the images of the geodesics of $G$. 
Thus by the proof of Proposition \ref{Pr:metrics_coincide} the following set contains precisely one element  that is $(G,G)$,
$$
\begin{array}{ll}
\{(h,h') \in \Met(R^{E}_{\p M}(\overline{M}))^2:& \partial^\alpha h_{ij} |_{R^{E}_{\p M}(\p M)}= \partial^\alpha h'_{ij} |_{R^{E}_{\p M}(\p M)} \hbox{ for every } \alpha \in \N^n,
\\
& \hbox{metrics  $h$ and $h'$ are geodesically equivalent,}
\\
& \hbox{sets $\overline \Sigma(p,\eta), \: (p,\eta) \in \p_-SM$ are the images} 
\\
& \hbox{of their geodesics}\}.
\end{array}
$$ 
We conclude that we have proven the following Proposition.

\begin{Proposition}
\label{Pr:reconstruction_of_G}
The data \eqref{bigger geodesic measurment data} determine the metric tensor $ G:= ((R^{E}_{\p M})^{-1})^{\ast}g$.
\end{Proposition}
%

Finally we note that the proof of Theorem \ref{th:reconstruction} follows from the Propositions \ref{pr:R_F is homeo}, \ref{Pr:reconstruction_of_smooth_struct} and \ref{Pr:reconstruction_of_G}.
\color{black}
\bigskip
 
\noindent
\textbf{Acknowledgements.} The research of ML, and TS was partly supported by  the Finnish Centre of Excellence in Inverse Problems Research
and Academy of Finland. In particular, ML was supported by projects 284715 and 303754, TS by projects 273979 and 263235. TS would like to express his gratitude to the Department of Pure Mathematics and Mathematical Statistics of the University of Cambridge for hosting him during two research visits when part of this work was carried out. TS would like to  also thank Prof. Gunther Uhlmann for hosting him in University of Washington where part of this work was carried out. HZ is supported by EPSRC grant 
EP/M023842/1. He would like to thank Prof. Gunther Uhlmann and the 
Department of Mathematics and Statistics of the University of Helsinki 
for the generous support of his visit to Helsinki during his PhD when 
part of the work was carried out.

\bibliographystyle{abbrv}
\bibliography{bibliography}

\end{document}